\documentclass[11pt, DIV=15]{scrartcl}
\usepackage[T1]{fontenc}
\usepackage[english]{babel}
\usepackage{amsmath, amsfonts, amsthm, amssymb, mathtools, stmaryrd}
\usepackage[pdfusetitle,hidelinks]{hyperref}
\usepackage[capitalise,noabbrev]{cleveref}
\usepackage{orcidlink}
\usepackage{csquotes}
\usepackage[osf,sc]{mathpazo}
\usepackage{relsize}
\usepackage{microtype}
\usepackage{multicol}
\usepackage{pdfpages}

\usepackage{thm-restate}
\recalctypearea

\usepackage{graphicx}

\usepackage{abstract}
\usepackage[inline]{enumitem}
\usepackage{booktabs}
\usepackage{todonotes}

\usepackage{tikz}
\usetikzlibrary{cd}
\usepackage{microtype}

\theoremstyle{definition}
\newtheorem{definition}{Definition}

\theoremstyle{plain}
\newtheorem{lemma}[definition]{Lemma}
\newtheorem{example}[definition]{Example}
\newtheorem{corollary}[definition]{Corollary}
\newtheorem{theorem}[definition]{Theorem}

\newtheorem{fact}[definition]{Fact}
\Crefname{fact}{Fact}{Facts}

\theoremstyle{remark}
\newtheorem{claim}{Claim}
\Crefname{claim}{Claim}{Claims}
\newenvironment{claimproof}[1][Proof of Claim]{\begin{proof}[#1] }{ \end{proof}}
\setlist[enumerate, 1]{font=\upshape, noitemsep, nolistsep}
\setlist[enumerate, 2]{font=\upshape, noitemsep, nolistsep}
\setlist[itemize, 1]{noitemsep, nolistsep,font=\upshape}
\setlist[itemize, 2]{noitemsep, nolistsep,font=\upshape}

\usepackage{amsmath, amsfonts, amsthm, amssymb, mathtools, stmaryrd}

\usepackage[hidelinks]{hyperref}
\usepackage[capitalise]{cleveref}

\usepackage{thm-restate}
\usepackage{relsize}

\usepackage{csquotes}
\usepackage{tikz}

\DeclareMathOperator{\spn}{span}

\newcommand{\HomInd}{\textup{\textsc{HomInd}}}
\newcommand{\Mem}{\textup{\textsc{NotIn}}}

\renewcommand\phi\varphi

\DeclareMathOperator{\cl}{cl}

\renewcommand\epsilon\varepsilon

\title{Logical~Equivalences, Homomorphism~Indistinguishability, and Forbidden~Minors\footnote{This work was presented at the \textit{48th International Symposium on Mathematical Foundations of Computer Science} (\textsmaller{MFCS}~2023)~\cite{seppelt_logical_2023}.} } 

\author{Tim Seppelt \orcidlink{0000-0002-6447-0568} \\ \small RWTH Aachen University \\\small
	\href{mailto:seppelt@cs.rwth-aachen.de}{\texttt{seppelt@cs.rwth-aachen.de}}}

\begin{document}
	\maketitle
	\begin{abstract}
		Two graphs $G$ and $H$ are \emph{homomorphism indistinguishable} over a class of graphs~$\mathcal{F}$ if for all graphs $F \in \mathcal{F}$ the number of homomorphisms from~$F$ to~$G$ is equal to the number of homomorphisms from~$F$ to~$H$.
		Many natural equivalence relations comparing graphs such as (quantum) isomorphism, spectral, and logical equivalences can be characterised as homomorphism indistinguishability relations over certain graph classes.
		
		Abstracting from the wealth of such instances, we show in this paper that equivalences w.r.t.\@ any \emph{self-complementarity} logic admitting a characterisation as homomorphism indistinguishability relation can be characterised by homomorphism indistinguishability over a minor-closed graph class.
		Self-complementarity is a mild property satisfied by most well-studied logics.
		This result follows from a correspondence between closure properties of a graph class and preservation properties of its homomorphism indistinguishability relation.
		
		Furthermore, we classify all graph classes which are in a sense finite (\emph{essentially profinite}) and satisfy the maximality condition of being \emph{homomorphism distinguishing closed}, i.e.\@ adding any graph to the class strictly refines its homomorphism indistinguishability relation.
		Thereby, we answer various questions raised by Roberson~(2022) on general properties of the homomorphism distinguishing closure.
\end{abstract}

	\section{Introduction}
	\label{sec:intro}

	In 1967, Lovász~\cite{lovasz_operations_1967} proved that two graphs $G$ and $H$ are isomorphic if and only if they are \emph{homomorphism indistinguishable} over all graphs, i.e.\@ for every graph $F$, the number of homomorphisms from~$F$ to~$G$ is equal to the number of homomorphisms from~$F$ to~$H$.
	Since then, homomorphism indistinguishability over restricted graph classes has emerged as a powerful framework for capturing a wide range of equivalence relations comparing graphs.
	For example, two graphs have cospectral adjacency matrices iff they are homomorphism indistinguishable over all cycles, cf.\@ \cite{dell_lovasz_2018}.
	They are quantum isomorphic iff they are homomorphism indistinguishable over all planar graphs \cite{mancinska_quantum_2019}.

	Most notably, equivalences with respect to many logic fragments can be characterised as homomorphism indistinguishability relations over certain graph classes \cite{grohe_counting_2020,dawar_lovasz-type_2021,montacute_pebble_2022,rattan_weisfeiler_2023}. 
	For example, two graphs satisfy the same sentences of $k$-variable counting logic iff they are homomorphism indistinguishable over graphs of treewidth less than $k$ \cite{dvorak_recognizing_2010}. 
	All graph classes featured in such characterisations are minor-closed and hence of a particularly enjoyable structure.
	The main result of this paper asserts that this is not a mere coincidence:
	In fact, logical equivalences and homomorphism indistinguishability over minor-closed graph classes are intimately related. 
	
	To make this statement precise, the term `logic' has to be formalised.	
	Following~\cite{barwise_extended_2017}, a \emph{logic on graphs} is a pair $(\mathsf{L}, \models)$ of a class $\mathsf{L}$ of \emph{sentences} and an isomorphism-invariant\footnote{For all $\phi \in \mathsf{L}$ and graphs $G$ and $H$ such that $G \cong H$, it holds that $G \models \phi$ iff $H \models \phi$.} \emph{model relation} $\models$ between graphs and sentences. 
	Two graphs $G$ and $H$ are \emph{$\mathsf{L}$-equivalent} if $G \models \phi$ iff $ H \models \phi$ for all $\phi \in \mathsf{L}$.
	One may think of a logic on graphs as a collection of isomorphism-invariant graph properties.
	We call a logic \emph{self-complementary} if for every $\phi \in \mathsf{L}$ there is an element $\overline{\phi} \in \mathsf{L}$ such that $G \models \phi$ if and only if $\overline{G} \models \overline{\phi}$.
	Here, $\overline{G}$ denotes the complement graph of $G$.
	Roughly speaking, a fragment/extension $\mathsf{L}$ of first-order logic is self-complementary if expressions of the form $Exy$ can be replaced by $\neg Exy \land (x \neq y)$ in every formula while remaining in $\mathsf{L}$.
	This lax requirement is satisfied by many logics including first-order logic, counting logic, second-order logic, fixed-point logics, and bounded variable, quantifier depth, or quantifier prefix fragments of these.
	All these examples are subject to the following result:
	
	\begin{theorem}[restate=thmLogic,name=]
		\label{thm:main1}
		Let $(\mathsf{L}, \models)$ be a self-complementary logic on graphs for which there exists a graph class
		$\mathcal{F}$ such that two graphs $G$ and $H$ are homomorphism indistinguishable over $\mathcal{F}$ if and only if they are $\mathsf{L}$-equivalent.
		Then there exists a minor-closed graph class $\mathcal{F}'$ whose homomorphism indistinguishability relation coincides with $\mathsf{L}$-equivalence.
	\end{theorem}

	\Cref{thm:main1} can be used to rule out that a given logic has a homomorphism indistinguishability characterisation (\cref{cor:complement}).
	Furthermore, it allows to use deep results of graph minor theory to study the expressive power of logics on graphs (\cref{cor:quantum}).
	
	\cref{thm:main1} is product of a more fundamental study of the properties of homomorphism indistinguishability relations. 
	In several instances, we show that closure properties of a graph class~$\mathcal{F}$ correspond to preservation properties of its homomorphism indistinguishability relation~$\equiv_{\mathcal{F}}$.
	These efforts yield answers to several open questions from \cite{roberson_oddomorphisms_2022}.
	A prototypical result is \cref{thm:main2}, from which \cref{thm:main1} follows. 
	For further results in the same vein, see \cref{tab:overview}.
	
	\begin{table}
		\centering
		\begin{tabular}{| l l l |}
			\hline
			\textbf{Closure property of $\mathcal{F}$} & \textbf{Preservation property of $\equiv_{\mathcal{F}}$} & \textbf{Theorem}\\ \hline
			taking minors   & complements      & \cref{thm:complement} \\  
			taking summands & disjoint unions   & \cref{thm:taking-summands} \\
			taking subgraphs& full complements & \cref{thm:full-complement} \\
			taking induced subgraphs & left lexicographic products  & \cref{prop:lexprod-indsub} \\ 
			contracting edges & right lexicographic products & \cref{prop:lexprod-contract} \\
			\hline
		\end{tabular}
		
		\caption{Overview of results on equivalent properties of a homomorphism distinguishing closed graph class $\mathcal{F}$ and of its homomorphism indistinguishability relation $\equiv_{\mathcal{F}}$.}
		\label{tab:overview}
	\end{table}
	
	\begin{theorem} \label{thm:main2}
		Let $\mathcal{F}$ be a homomorphism distinguishing closed graph class.
		Then $\mathcal{F}$ is minor-closed if and only if $\equiv_{\mathcal{F}}$ is preserved under taking complements, i.e.\@ for all simple graphs $G$ and $H$ it holds that $G \equiv_{\mathcal{F}} H$ if and only if $\overline{G} \equiv_{\mathcal{F}} \overline{H}$.
	\end{theorem}
	
	Here, the graph class $\mathcal{F}$ is \emph{homomorphism distinguishing closed} \cite{roberson_oddomorphisms_2022} if for every $K \not\in \mathcal{F}$ there exist graphs $G$ and $H$ which are homomorphism indistinguishable over $\mathcal{F}$ but differ in the number of homomorphisms from~$K$. In other words, adding even a single graph to $\mathcal{F}$ would change its homomorphism indistinguishability relation.
	
	Proving that a graph class is homomorphism distinguishing closed is a pathway to separating equivalence relations comparing graphs \cite{roberson_lasserre_2023}.
	However, establishing this property is a notoriously hard task.
	Thus, a general result establishing the homomorphism distinguishing closedness of a wide range of graph classes would be desirable.
	In \cite{roberson_oddomorphisms_2022}, Roberson conjectured that \emph{every graph class closed under taking minors and disjoint unions is homomorphism distinguishing closed}.
	At present, this conjecture has only been  verified for few graph classes~\cite{roberson_oddomorphisms_2022,neuen_homomorphism-distinguishing_2023}.

	Our final result confirms Roberson's conjecture for all graph classes which are in a certain sense finite.
	The expressive power of homomorphism counts from finitely many graphs is of particular importance in practice. Applications include the design of graph kernels \cite{KriegeJM20}, motif counting \cite{AlonDHHS08,milo2002network}, or machine learning on graphs \cite{BeaujeanSY21,NguyenM20,Grohe-x2vec}.
	A theoretical interest stems for example from database theory where homomorphism counts correspond to results of queries under bag-semantics \cite{chaudhuri_optimization_1993,kwiecien_determinacy_2022}, see also \cite{chen_algorithms_2022}.
	
	Since every homomorphism distinguishing closed graph class is closed under taking disjoint unions, infinite graph classes arise inevitably when studying homomorphism indistinguishability over finite graph classes.
	We introduce the notions of \emph{essentially finite} and \emph{essentially profinite} graph classes (\cref{def:profinite}) in order to capture the nevertheless limited behaviour of graph classes arising from the finite.
	Examples for essentially profinite graph classes include the class of all minors of a fixed graph and the class of cluster graphs, i.e.\@ disjoint unions of arbitrarily large cliques.
	In \cref{thm:profinite}, the essentially profinite graph classes which are homomorphism distinguishing closed are fully classified. 
	Thereby, the realm of available examples of homomorphism distinguishing closed graph classes is drastically enlarged.
	This classification has the following readily-stated corollary:

	\begin{theorem} \label{thm:main3}
		Every essentially profinite union-closed graph class $\mathcal{F}$ which is closed under taking summands\footnote{A graph class $\mathcal{F}$ is \emph{closed under taking summands} if for all $F \in \mathcal{F}$ which is the disjoint union of two graphs $F_1 + F_2 = F$ also $F_1, F_2 \in \mathcal{F}$.} is homomorphism distinguishing closed.
		In particular, Roberson's conjecture holds for all essentially profinite graph classes.
	\end{theorem}

	\section{Preliminaries}
	
	All graphs in this article are finite, undirected, and without multiple edges. 
	A \emph{simple graph} is a graph without loops.
	A \emph{homomorphism} from a graph $F$ to a graph $G$ is a map $h \colon V(F) \to V(G)$ such that $h(u)h(v) \in E(G)$ whenever $uv \in E(F)$ and vertices carrying a loop are mapped to vertices carrying a loop. Write $\hom(F, G)$ for the number of homomorphisms from $F$ to $G$. 
	For a class of graphs $\mathcal{F}$ and graphs $G$ and $H$, write $G \equiv_{\mathcal{F}} H$ if $\hom(F, G) = \hom(F, H)$ for all $F \in \mathcal{F}$, i.e.\@ $G$ and $H$ are \emph{homomorphism indistinguishable over $\mathcal{F}$}.
	With the exception of \cref{sec:full-complements}, the graphs in $\mathcal{F}$ as well as $G$ and $H$ will be simple.
	Following \cite{roberson_oddomorphisms_2022}, the \emph{homomorphism distinguishing closure} of $\mathcal{F}$ is
	\[
		\cl(\mathcal{F}) \coloneqq \{K \text{ simple graph} \mid \forall  \text{ simple graphs } G, H.\quad G \equiv_{\mathcal{F}} H \Rightarrow \hom(K, G) = \hom(K, H) \}.
	\]
	Intuitively, $\cl(\mathcal{F})$ is the \enquote{largest} graph class whose homomorphism indistinguishability relation coincides with the one of $\mathcal{F}$.
	A graph class $\mathcal{F}$ is \emph{homomorphism distinguishing closed} if $\cl(\mathcal{F}) = \mathcal{F}$. Note that $\cl$ is a closure operator in the sense that $\cl(\mathcal{F}) \subseteq \cl(\mathcal{F}')$ if $\mathcal{F} \subseteq \mathcal{F}'$ and $\cl(\cl(\mathcal{F})) = \cl(\mathcal{F})$ for all graph classes $\mathcal{F}$ and $\mathcal{F}'$.
	Furthermore, the following inclusions hold (and cannot be reversed, cf.\@ \cref{ex1}):
	
	\begin{lemma} \label{lem:intersection}
		Let $I$ be an arbitrary index set and let $(\mathcal{F}_i)_{i \in I}$ be a family of graph classes.
		Then 
		\[
			\cl\left(\bigcap_{i\in I} \mathcal{F}_i \right) \subseteq \bigcap_{i \in I} \cl(\mathcal{F}_i) \quad \text{and} \quad
			\bigcup_{i\in I} \cl(\mathcal{F}_i) \subseteq \cl\left( \bigcup_{i\in I} \mathcal{F}_i \right).
		\]
	\end{lemma}
	\begin{proof}
		Any intersection of homomorphism distinguishing closed graph classes is homomorphism distinguishing closed \cite[Lemma~6.1]{roberson_oddomorphisms_2022}. Hence, $\bigcap_{i \in I} \cl(\mathcal{F}_i)$ is closed and it suffices to observe that $\bigcap_{i \in I} \mathcal{F}_i \subseteq \bigcap_{i \in I} \cl(\mathcal{F}_i)$.
		For the second claim, $\mathcal{F}_j \subseteq \bigcup_{i\in I} \mathcal{F}_i$ and hence $\cl(\mathcal{F}_j) \subseteq \cl(\bigcup_{i\in I} \mathcal{F}_i)$ for $j\in I$.
	\end{proof}
	
	For graphs $F$ and $K$, $F$ is said to be \emph{$K$-colourable} if there is a homomorphisms $F \to K$. $F$ and $K$ are \emph{homomorphically equivalent} if there are homomorphisms $F \to K$ and $K \to F$.
	
	For graphs $G$ and $H$, write $G + H$ for their \emph{disjoint union} and $\coprod_{i \in [n]} G_i \coloneqq G_1 + \dots + G_n$ for graphs $G_1, \dots, G_n$.
We write $G \times H$ for the \emph{categorical product} of $G$ and $H$, i.e.\@ $V(G \times H) \coloneqq V(G) \times V(H)$ and $gh$ and $g'h'$ are adjacent in $G \times H$ iff $gg' \in E(G)$ and $hh' \in E(H)$. The \emph{lexicographic product} $G \cdot H$ is defined as the graph with vertex set $V(G) \times V(H)$ and edges between $gh$ and $g'h'$ iff $g = g'$ and $hh' \in E(H)$ or $gg' \in E(G)$. It is well-known, cf.\@ e.g.\@ \cite[(5.28)--(5.30)]{lovasz_large_2012}, that
	for all graphs $F_1, F_2, G_1, G_2$, and all connected graphs $K$, the following equalities hold:
	\begin{align}
		\hom(F_1 + F_2, G) &= \hom(F_1, G) \hom(F_2, G), \label{eq:coproduct} \\
		\hom(F, G_1 \times G_2) &= \hom(F, G_1) \hom(F, G_2), \text{ and }\label{eq:product} \\
		\hom(K, G_1 + G_2) &= \hom(K, G_1) + \hom(K, G_2). \label{eq:disjoint}
	\end{align}
		
	The \emph{complement} of a simple graph $F$ is the simple graph $\overline{F}$ with $V(\overline{F}) = V(F)$ and $E(\overline{F}) = \binom{V(F)}{2} \setminus E(F)$.
	Observe that $\overline{\overline{F}} = F$ for all simple graphs $F$.
	The \emph{full complement} of a  graph $G$ is the graph $\widehat{G}$ obtained from $G$ by replacing every edge with a non-edge and every loop with a non-loop, and vice-versa.
	
	The \emph{quotient} $F/\mathcal{P}$ of a simple graph $F$ by a partition $\mathcal{P}$ of $V(F)$ is the simple graph with vertex set $\mathcal{P}$ and edges $PQ$ for $P \neq Q$ iff there exist vertices $p \in P$ and $q \in Q$ such that $pq \in E(F)$. 
For a set $X$, write $\Pi(X)$ for the set of all partitions of $X$.	
	A graph $F'$ can be obtained from a simple graph $F$ by \emph{contracting edges} if there is a partition $\mathcal{P} \in \Pi(V(F))$ such that $F[P]$ is connected for all $P \in \mathcal{P}$ and $F' \cong F/\mathcal{P}$.
	
	For a graph $F$ and a set $P \subseteq V(F)$, write $F[P]$ for the \emph{subgraph induced by $P$}, i.e.\@ the graph with vertex set $P$ and edges $uv$ if $u,v \in P$ and $uv \in E(F)$.
	A graph $F'$ is a \emph{subgraph} of $F$, in symbols $F' \subseteq F$ if $V(F') \subseteq V(F)$ and $E(F') \subseteq E(F)$.
	A \emph{minor} of a simple graph $F$ is a subgraph of a graph which can be obtained from $F$ by contracting edges. 
	
	\begin{fact}[Inclusion--Exclusion] \label{fact:incl-excl}
		Let $A_1, \dots, A_n$ be a subsets of a finite set $S$. For $A \subseteq S$, write $\overline{A} \coloneqq S \setminus A$. Then
		\[
		\left\lvert \bigcap_{i \in [n]} \overline{A_i} \right\rvert = |S| +  \sum_{\emptyset \neq I \subseteq [n]} (-1)^{|I|} \left\lvert \bigcap_{i \in I} A_i \right\rvert.
		\]
	\end{fact}

	\section{Closure Properties Correspond to Preservation Properties}
	\label{sec:syntax-semantics}

	This section is concerned with the interplay of closure properties of a graph class $\mathcal{F}$ and preservation properties of its homomorphism indistinguishability relation $\equiv_{\mathcal{F}}$. The central results of this section are \cref{thm:main2} and the other results listed in \cref{tab:overview}.  
	
	The relevance of the results is twofold: On the one hand, they yield that if a graph class $\mathcal{F}$ has a certain closure property then so does $\cl(\mathcal{F})$. In the case of minor-closed graph families, this provides evidence for Roberson's conjecture \cite{roberson_oddomorphisms_2022}. On the other hand, they establish that equivalence relations comparing graphs which are preserved under certain operations coincide with the homomorphism indistinguishability relation over a graph class with a certain closure property, if they are homomorphism indistinguishability relations at all. 
	Further consequences are discussed in \cref{sec:applications,sec:logic}.
	Essential to all proofs is the following \cref{lem:lincomb}:
	
	\begin{lemma} \label{lem:lincomb}
		Let $\mathcal{F}$ and $\mathcal{L}$ be classes of simple graphs.
		Suppose $\mathcal{L}$ is finite 
		and that its elements are pairwise non-isomorphic.
		Let $\alpha \colon \mathcal{L} \to \mathbb{R} \setminus \{0\}$. If for all simple graphs $G$ and $H$
		\[
		G \equiv_{\mathcal{F}} H \implies \sum_{L \in \mathcal{L}} \alpha_L \hom(L, G) = \sum_{L \in \mathcal{L}} \alpha_L \hom(L, H)
		\]
		then $\mathcal{L} \subseteq \cl(\mathcal{F})$.
	\end{lemma}

	\begin{proof}
		The following argument is due to \cite[Lemma~3.6]{curticapean_homomorphisms_2017}. 
		Let $n$ be an upper bound on the number of vertices of graphs in $\mathcal{L}$ and let $\mathcal{L}'$ denote the class of all graphs on at most $n$~vertices. 
		By classical arguments~\cite{lovasz_operations_1967}, cf.\@ \cite[Proposition~5.44(b)]{lovasz_large_2012}, the matrix $M \coloneqq (\hom(K, L))_{K, L \in \mathcal{L}'}$ is invertible. Extend $\alpha$ to a function $\alpha' \colon \mathcal{L}' \to \mathbb{R}$ by setting $\alpha'(L) \coloneqq \alpha(L)$ for all $L \in \mathcal{L}$ and $\alpha'(L') \coloneqq 0$ for all $L' \in \mathcal{L}' \setminus \mathcal{L}$.
		By \cref{eq:product}, if $G \equiv_{\mathcal{F}} H$ then $G \times K \equiv_{\mathcal{F}} H \times K$ for all graphs~$K$. Hence, by \cref{eq:product},
		\[
		\sum_{L \in \mathcal{L'}} \hom(L, K) \alpha_L \hom(L, G) = \sum_{L \in \mathcal{L'}} \hom(L, K) \alpha_L \hom(L, H).
		\]
		Both sides can be read as the product of the matrix $M^T$ with a vector of the form $(\alpha_L \hom(L, -))_{L \in \mathcal{L'}}$. 
		By multiplying from the left with the inverse of $M^T$,
		it follows that $\alpha_L \hom(L, G) = \alpha_L \hom(L, H)$ for all $L \in \mathcal{L}'$ which in turn implies that $\hom(L, G) = \hom(L, H)$ for all $L \in \mathcal{L}$. Thus, $\mathcal{L} \subseteq \cl(\mathcal{F})$.
	\end{proof}

	In the setting of \cref{lem:lincomb}, we say that the relation $\equiv_{\mathcal{F}}$ \emph{determines} the linear combination $\sum_{L \in \mathcal{L}} \alpha_L \hom(L, -)$. Note that it is essential for the argument to carry through that the elements of $\mathcal{L}$ are pairwise non-isomorphic and that $\alpha_L \neq 0$ for all $L$. Efforts will be undertaken to establish this property for certain linear combinations in the subsequent sections.

	\subsection{Taking Summands and Preservation under Disjoint Unions}
	\label{sec:summands}
	\begin{figure}
		\centering
		\begin{tikzpicture}[node distance=3cm]
			\node (minors) {minors};
			\node [right of=minors, yshift=-.5cm] (edgdel) {edge deletion};
			\node [right of=edgdel] (subgraphs) {subgraphs};
			\node [right of=subgraphs] (indsub) {induced subgraphs};
			\node [right of=indsub, yshift=.5cm] (summands) {summands};
			
			\node [above of=subgraphs, yshift=-2cm] (contr) {edge contractions};
			
			\draw [->] (minors) |- (contr);
			\draw [->, dashed] (contr) -| (summands);
			\draw [->] (minors) |- (edgdel);
			\draw [->, dashed] (edgdel) -- (subgraphs);
			\draw [->] (subgraphs) -- (indsub);
			\draw [->] (indsub) -| (summands);
		\end{tikzpicture}
		\caption{Relationship between closure properties of homomorphism distinguishing closed graph classes. The non-obvious implications are dashed and proven in \cref{lem:minors,lem:contr}.}
		\label{fig:relationship}
	\end{figure}
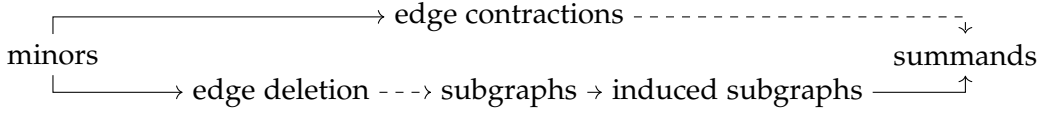

	In this section, the strategy yielding the results in \cref{tab:overview} is presented for the rather simple case of \cref{thm:taking-summands}. This theorem relates the property of a graph class $\mathcal{F}$ to be closed under taking summands to the property of $\equiv_{\mathcal{F}}$ to be preserved under disjoint unions. 
	This closure property is often assumed in the context of homomorphism indistinguishability \cite{abramsky_discrete_2022,roberson_oddomorphisms_2022} and fairly mild. It is the most general property among those studied here, cf.\@ \cref{fig:relationship}.
	\cref{thm:taking-summands} answers a question from~\cite[p.\@~7]{roberson_oddomorphisms_2022} affirmatively: Is it true that if $\equiv_{\mathcal{F}}$ is preserved under disjoint unions then $\cl(\mathcal{F})$ is closed under taking summands?
	
	\begin{theorem}  \label{thm:taking-summands}
		For a graph class $\mathcal{F}$ and the assertions
		\begin{enumerate}
			\item $\mathcal{F}$ is \emph{closed under taking summands}, i.e.\@ if $F_1 + F_2 \in \mathcal{F}$ then $F_1, F_2 \in \mathcal{F}$,\label{ts1}
			\item $\equiv_{\mathcal{F}}$ is \emph{preserved under disjoint unions}, i.e.\@ for all simple graphs $G$, $G'$, $H$, and $H'$, if $G \equiv_{\mathcal{F}} G'$ and $H \equiv_{\mathcal{F}} H'$ then $G + H \equiv_{\mathcal{F}} G' + H'$,\label{ts2}
			\item $\cl(\mathcal{F})$ is closed under taking summands,\label{ts3}
		\end{enumerate}
		the implications \ref{ts1} $\Rightarrow$ \ref{ts2} $\Leftrightarrow$ \ref{ts3} hold.
	\end{theorem}
	\begin{proof}
		The central idea is to write, given graphs $F$, $G$, and $H$, the quantity $\hom(F, G + H)$ as expression in $\hom(F', G)$ and $\hom(F', H)$ where the $F'$ range over summands of $F$.
		To this end, write $F = C_1 + \dots + C_r$ as disjoint union of its connected components. Then,
		\begin{align}
			\hom(F, G + H) 
			&\overset{\eqref{eq:coproduct}}{=} \prod_{i=1}^r \hom(C_i, G + H) \notag
			\overset{\eqref{eq:disjoint}}{=} \prod_{i=1}^r \left( \hom(C_i, G) + \hom(C_i, H) \right) \notag \\
			&\overset{\eqref{eq:coproduct}}{=} \sum_{I \subseteq [r]} \hom(\coprod_{i \in I} C_i, G) \hom(\coprod_{i \in [r] \setminus I} C_i, H). \label{eq:disjunion}
		\end{align}
		In particular, if $\mathcal{F}$ is closed under taking summands 
		then $\coprod_{i \in I} C_i \in \mathcal{F}$ for all $I \subseteq [r]$.
		Thus, \ref{ts1} implies \ref{ts2}.
		
		Assume \ref{ts2} and let $F \in \cl(\mathcal{F})$. Write as above $F = C_1 + \dots + C_r$ as disjoint union of its connected components.
		By the assumption that $\equiv_{\mathcal{F}}$ is preserved under disjoint unions, for all graphs $G$ and $G'$, if $G \equiv_{\mathcal{F}} G'$ then $G + F \equiv_{\mathcal{F}} G' + F$ and hence $\hom(F, G+F) = \hom(F, G' + F)$.
		By \cref{eq:disjunion} with $H = F$, 
		the relation $\equiv_{\mathcal{F}}$ determines the linear combination $\sum_{I \subseteq [r]} \hom(\coprod_{i \in I} C_i, -) \hom(\coprod_{i \in [r] \setminus I} C_i, F)$. 
		Note that it might be the case that $\coprod_{i \in I} C_i \cong \coprod_{j \in J} C_j$ for some $I \neq J$.
		Grouping such summands together and adding their coefficients yields a linear combination satisfying the assumptions of \cref{lem:lincomb} since $\hom(C_i, F) > 0$ for all $i \in [r]$. Hence, $\coprod_{i \in I} C_i \in \cl(\mathcal{F})$ for all $I \subseteq [r]$ and \ref{ts3} follows.

		The implication \ref{ts3} $\Rightarrow$ \ref{ts2} follows from \ref{ts1} $\Rightarrow$ \ref{ts2} for $\cl(\mathcal{F})$ since $\equiv_{\mathcal{F}}$ and $\equiv_{\cl(\mathcal{F})}$ coincide.
	\end{proof}

	The proofs of the other results in \cref{tab:overview} are conceptually similar to the just completed proof. 
	The general idea can be briefly described as follows:
	\begin{enumerate}
	\item Derive a linear expression similar to \cref{eq:disjunion} for the number of homomorphisms from $F$ into the graph constructed using the assumed preservation property of $\equiv_{\mathcal{F}}$, e.g.\@ the graph $G+H$ in the case of \cref{thm:taking-summands}. These linear combinations typically involve sums over subsets $U$ of vertices or edges of~$F$, each contributing a summand of the form $\alpha_U\hom(F_U, -)$ where $\alpha_U$ is some coefficient and $F_U$ is a graph constructed from $F$ using~$U$. Hence, if $\mathcal{F}$ is closed under the construction transforming $F$ to $F_U$ then $\equiv_{\mathcal{F}}$ has the desired preservation property.
	\item In general, it can be that $F_U$ and $F_{U'}$ are isomorphic even though $U \neq U'$, e.g.\@ in \cref{eq:disjunion} if $F$ contains two isomorphic connected components. 
	In order to apply \cref{lem:lincomb}, one must group the summands $\alpha_U\hom(F_U, -)$ by the isomorphism type $F'$ of the $F_U$. 
	The coefficient of $\hom(F', -)$ in the new linear combination ranging over pairwise non-isomorphic graphs is the sum of $\alpha_U$ over all $U$ such that $F_U \cong F'$.
	Once it is established that this coefficient is non-zero, it follows that if $\equiv_{\mathcal{F}}$ has the preservation property then $\cl(\mathcal{F})$ has the desired closure property.
	\end{enumerate}

	\subsection{Taking Minors and Preservation under Complements}
	\label{sec:complements}
	
	The strategy outlined in the previous section is now applied to prove \cref{thm:complement}, which implies \cref{thm:main2}.
	This answers a question of Roberson  \cite[Question~8]{roberson_oddomorphisms_2022} affirmatively: Is it true that if $\mathcal{F}$ is such that $\equiv_{\mathcal{F}}$ is preserved under taking complements then there exist a minor-closed $\mathcal{F'}$ such that $\equiv_{\mathcal{F}}$ and $\equiv_{\mathcal{F}'}$ coincide.
	
	\cref{thm:complement} is among the first results substantiating Roberson's conjecture not only for example classes but in full generality.
	In particular, as noted in \cite[p.\@~2]{roberson_oddomorphisms_2022}, there was little reason to believe that minor-closed graph families should play a distinct role in the theory of homomorphism indistinguishability. \cref{thm:complement} indicates that this might be the case.
	Indeed, while Roberson's conjecture asserts that $\cl(\mathcal{F})$ coincides with~$ \mathcal{F}$ for every minor-closed and union-closed graph class~$\mathcal{F}$, \cref{thm:complement} yields unconditionally that $\cl(\mathcal{F})$ is a minor-closed graph class itself.
	
	\begin{theorem}
		\label{thm:complement}
		For a graph class $\mathcal{F}$ and the assertions
		\begin{enumerate}
			\item $\mathcal{F}$ is closed under edge contraction and deletion,\label{rcl1}
			\item $\equiv_{\mathcal{F}}$ is \emph{preserved under taking complements}, i.e.\@ for all simple graphs $G$ and $H$ it holds that $G \equiv_{\mathcal{F}} H$ if and only if $\overline{G} \equiv_{\mathcal{F}} \overline{H}$,\label{rcl2}
\item $\cl(\mathcal{F})$ is minor-closed,\label{rcl4}
		\end{enumerate}
		the implications \ref{rcl1} $\Rightarrow$ \ref{rcl2} $\Leftrightarrow$ \ref{rcl4} hold.
	\end{theorem}

	The strategy is to write $\hom(F, \overline{G})$ as a linear combination of $\hom(F', G)$ for minors $F'$ of $F$.
	This is accomplished in two steps: First, we consider the full complement $\widehat{G}$ of $G$ in which not only every edge of~$G$ is replaced by a non-edge (and vice versa) but also every loop is replaced by a non-loop (and vice versa). Secondly, we consider for a simple graph $G$ the looped graph $G^\circ$ obtained from $G$ by adding a loop to every vertex.
	The fact that $\widehat{G^\circ} \cong \overline{G}$ motivates this two-step approach. 
	
	\begin{lemma}[{\cite[Equation~(5.23)]{lovasz_large_2012}}] \label{eq:complement}
		For every  graph $G$ and every simple graph $F$,
		\[
		\hom(F, \widehat{G}) = \sum_{F' \subseteq F \text{ s.t.\@ }V(F') = V(F)} (-1)^{|E(F')|} \hom(F', G).
		\]
	\end{lemma}
	\begin{proof}Write $S$ for the set of all maps $V(F) \to V(G)$. For $e \in E(F)$, write $A_e \subseteq S$ for set of all maps $h$ such that the image of $e$ under $h$ is an edge or a loop in $G$.
		It holds that $h \not\in A_e$ if and only if $e$ is mapped under $h$ either to two non-adjacent vertices or to a single vertex without a loop.
		Hence, $\hom(F, \widehat{G}) = \left\lvert \bigcap_{e \in E(F)} \overline{A_e} \right\rvert$.
		The claim follows from the Inclusion--Exclusion Principle (cf.\@ \cref{fact:incl-excl}) noting that $|S| = \hom(|F|K_1, G)$ and $\left\lvert \bigcap_{e \in F'} A_e \right\rvert$ for some $F' \subseteq F$ is the number of homomorphisms from the graph $F'$ to $G$ which is obtained by deleting all edges from $F$ which are not in $F'$.
	\end{proof}

	In light of \cref{eq:complement}, it suffices to write $\hom(F, G^\circ)$ as a linear combination of $\hom(F', G)$ for minors $F'$ of $F$.
	To ease bookkeeping, we consider a particular type of quotient graphs.
	For a simple graph $F$ and a set of edges $L \subseteq E(F)$, define the \emph{contraction relation}~$\sim_L$ on $V(F)$ by declaring $v \sim_L w$ if
	$v$ and $w$ lie in the same connected component of the subgraph of $F$ with vertex set $V(F)$ and edge set $L$.
	Write $[v]_L$ for the classes of $v \in V(F)$ under the equivalence relation $\sim_L$.
	
	The \emph{contraction quotient} $F \oslash L$ is the graph whose vertex set is the set of equivalence classes under $\sim_L$ and with an edge between $[v]_L$ and $[w]_L$ if and only if there is an edge $xy \in E(F) \setminus L$ such that $x \sim_L v$ and $y \sim_L w$. 
	In general, $F \oslash L$ may contain loops, cf.\@ \cref{ex:contr}. However, if it is simple then it is equal to $F/\mathcal{P}$ where $\mathcal{P}$ is the partition of $V(F)$ into equivalence classes under $\sim_L$, i.e.\@ $\mathcal{P} \coloneqq \{[v]_L \mid v \in V(F)\}$. In this case, $F \oslash L$ is a graph obtained from $F$ by edge contractions.
	
	With this notation, the quantity $\hom(F, G^\circ)$ can be succinctly written as linear combination.
	Using \cref{eq:complement,lem:looping}, \cref{thm:complement} can be proven by the strategy described in \cref{sec:summands}.
	
	\begin{theorem} \label{lem:looping}
		Let $F$ and $G$ be simple graphs. Then
		\[
		\hom(F, G^\circ) = \sum_{L \subseteq E(F)} \hom(F \oslash L, G).
		\]
	\end{theorem}
	\begin{proof}The proof is by establishing a bijection between the set of homomorphisms $h \colon F \to G^\circ$ and the set of pairs $(L, \phi)$ where $L \subseteq E(F)$ and $\phi \colon F \oslash L \to G$ is a homomorphism.	
	
	\begin{claim} \label{cl:looping-map}
		The map which associates a homomorphism $h \colon F \to G^\circ$ with the pair $(L, \phi)$ consisting of the set $L \coloneqq \{e \in E(F) \mid |h(e)| = 1\}$ and the homomorphism $\phi \colon F \oslash L \to G$, $[x]_L \mapsto h(x)$ is well-defined.
	\end{claim}
	\begin{claimproof}
		First observe that if $v,w  \in V(F)$ are such that $v \sim_L w$ then $h(v) = h(w)$. Indeed, if wlog $v \neq w$ then there exists a walk $vu_1, u_1u_2, \dots, u_nw \in L$. By construction, $|h(vu_1)| = |h(u_1u_2)| = \dots = |h(u_nw)| = 1$ and hence $h(v) = h(w)$.
		In particular, $F \oslash L$ is a simple graph without any loops. Indeed, if $vw \in E(F)$ are such that $v \sim_L w$ then $h(v) = h(w)$ and hence $vw \in L$.
		
		The initial observation implies that $\phi$ is a well-defined map $V(F) \to V(G)$. It remains to argue that $\phi$ is a homomorphism $F \to G$. Indeed, if $[v]_L \neq [w]_L$ are adjacent in $F \oslash L$ then there exist $x \sim_L v$ and $y \sim_L w$ such that $xy \in E(F) \setminus L$. Hence, $h$ maps $xy$ to an edge in $G^\circ$, rather than to a loop. In particular, $\phi$ maps $[v]_L$ and $[w]_L$ to an edge in $G$.
	\end{claimproof}
	
	\begin{claim}
		The map $h \mapsto (L, \phi)$ described in \cref{cl:looping-map} is surjective.
	\end{claim}
	\begin{claimproof}
		Let $L \subseteq E(F)$ and $\phi \colon F \oslash L \to G$ be a homomorphism. Observe that this implies that $F \oslash L$ is without loops. 
		Let $\pi \colon V(F) \to V(F \oslash L)$ denote the projection $v \mapsto [v]_L$.
		It is claimed that $h \coloneqq \phi \circ \pi$ is a homomorphism $F \to G^\circ$. Let $xy \in E(F)$. If $x \sim_L y$ then the image of $xy$ under $h$ is a loop since then $\pi(x) = \pi(y)$.
		If $x \not\sim_L y$ then in particular $xy \not\in L$ and there is an edge between $[x]_L$ and $[y]_L$ in $F \oslash L$. Since $\phi$ is a homomorphism, $h(x)$ and $h(y)$ are in both cases adjacent in~$G^\circ$.
		
		It remains to argue that this $h$ is mapped to $(L, \phi)$ under the construction described in \cref{cl:looping-map}.
		Write $L' \coloneqq \{e \in E(F) \mid |h(e)| = 1\}$. Towards concluding that $L = L'$, 
		let $uv \in L$. Then $h(uv) = \phi(\pi(uv))$ is a singleton since in particular $u \sim_L v$. Hence, $L \subseteq L'$.
		Conversely, let $uv \in L'$. By assumption, $h(uv) = \phi(\pi(uv))$ is a singleton and thus it remains to distinguish two cases:
		If $\pi(u) = \pi(v)$ then $uv \in L$ because otherwise there would be a loop at $[u]_L = [v]_L$ in $F \oslash L$.
		If $\pi(u) \neq \pi(v)$ then $\phi([u]_L) = \phi([v]_L)$ and also $uv \in L$ because otherwise there would be an edge between $[u]_L \neq [v]_L$ in $F \oslash L$ and this cannot happen since $\phi$ is a homomorphism into a simple graph.
		
		Finally, write $\phi'$ for the homomorphism $F \oslash L \to G$ constructed from $h$ as described in \cref{cl:looping-map}. Then for every vertex $x \in V(F)$ by definition, $\phi'([x]_L) = h(x) = \phi(\pi(x)) = \phi([x]_L)$ and thus $\phi = \phi'$.
	\end{claimproof}
	It is easy to see that the map devised in \cref{cl:looping-map} is injective.
	The desired equation follows observing that the map in \cref{cl:looping-map} provides a bijection between the sets whose elements are counted on the right and left hand-side respectively.
\end{proof}
		
	The following \cref{ex:contr} illustrates the above construction and \cref{lem:looping}.
	\begin{example}
		\label{ex:contr}
		Let $K_3$ denote the clique with vertex set $\{1,2,3\}$. Then $K_3 \oslash \emptyset \cong K_3$, $K_3 \oslash \{12\} \cong K_2$, $K_3 \oslash \{12,23\} \cong K_1^\circ$, and $K_3 \oslash \{12,23,13\} \cong K_1$.
		For every simple graph $G$,
		$\hom(K_3, G^\circ) = \hom(K_3, G) + 3 \hom(K_2, G) + \hom(K_1, G)$ since $\hom(K_1^\circ, G) = 0$.
	\end{example}
	
	The final ingredient for the proof of \cref{thm:complement} is the following \cref{lem:minors}, which establishes one of the implications in \cref{fig:relationship}.
	
	\begin{lemma} \label{lem:minors}
		If a homomorphism distinguishing closed graph class $\mathcal{F}$ is closed under deleting edges then it is closed under taking subgraphs.
	\end{lemma}
	\begin{proof}Let $F \in \mathcal{F}$ be a graph on $n$ vertices. Since the graph $nK_1$ can be obtained from $F$ by deleting all its edges, it holds that $nK_1 \in \mathcal{F}$.
		Thus, for all graphs $G$ and $H$ such that $G \equiv_{\mathcal{F}} H$, by \cref{eq:coproduct},
		$|V(G)|^n = \hom(nK_1, G) = \hom(nK_1, H)  = |V(H)|^n$.
		Hence, $\hom(K_1, G) = \hom(K_1, H)$.
		Now let $F'$ denote the graph obtained from $F$ by deleting a vertex $v \in V(F)$.
		By deleting all incident edges, it can be assumed that $v$ is isolated in $F$ and that $F' + K_1 \cong F$. 
		Hence, for all $G$ and $H$ as above,
		\[
		\hom(F', G)\hom(K_1, G) = \hom(F, G) = \hom(F, H) = \hom(F', H)\hom(K_1, H).
		\]
		It follows that $\hom(F', G) = \hom(F', H)$ and $F' \in \mathcal{F}$.
	\end{proof}

	\begin{proof}[Proof of \cref{thm:complement}]
		Assuming \ref{rcl1}, let $G$ and $H$ be such that $G \equiv_{\mathcal{F}} H$ and let $F \in \mathcal{F}$. By \cref{lem:looping,eq:complement},
		\begin{equation} \label{eq:del-contr}
			\hom(F, \overline{G}) 
			= \hom(F, \widehat{G^\circ})
			= \sum_{\substack{F' \subseteq F, \\ V(F') = V(F)}} (-1)^{|E(F')|}  \sum_{L \subseteq E(F')} \hom(F' \oslash L, G).
		\end{equation}
		All simple graphs $F' \oslash L$ appearing in this sum are obtained from $F$ by repeated edge contractions or deletions. Hence, $\overline{G} \equiv_{\mathcal{F}} \overline{H}$.
		
		Assuming \ref{rcl2}, it is first shown that $\cl(\mathcal{F})$ is closed under deleting and contracting edges. 
		To that end, let $F \in \cl(\mathcal{F})$. Let $K$ be obtained from $F$ by deleting an edge $e \in E(F)$.
		Since $K$ has the same number of vertices as $F$ and precisely one edge less, the only summation indices $(F', L)$ in \cref{eq:del-contr} such that $F' \oslash L \cong K$
		satisfy 
		$|E(F')| = |E(K)| - 1$ and $L = \emptyset$.
		Hence, the coefficient of $\hom(K, -)$ in the linear combination obtained from \cref{eq:del-contr} by grouping isomorphic summation indices is a non-zero multiple of $(-1)^{|E(K)|}$ and $K \in \cl(\mathcal{F})$ by \cref{lem:lincomb}.
		
		Let now $K$ be obtained from $F$ by contracting a single edge $e$.
		By deleting edges, it can be supposed without loss of generality that there are no vertices in $F$ which are adjacent to both end vertices of $e$.
		By the following \cref{obs:edges-contract}, $|E(K)| = |E(F)| - 1$.
		
			\begin{claim} \label{obs:edges-contract}
			For a simple graph $F$ with $uv \in E(F)$, write $\Delta_F(uv) \coloneqq \{w \in V(F) \mid uv, vw, uw \in E(F)\}$ for the set of vertices inducing a triangle with the edge $uv$.
			Let $F$ be a simple graph and $uv \in E(F)$. Then
			\[
			|\Delta_F(uv)| = |E(F)| - |E(F \oslash \{uv\})| - 1.
			\]
		\end{claim}
		\begin{claimproof}
			Write $K = F \oslash \{uv\}$.
			By writing $x$ for the vertex of $K$ obtained by contraction and identifying $V(K) \setminus \{x\} = V(F) \setminus \{u, v\}$, for all $w \in V(K)$,
			\[
			\deg_{K} w = \begin{cases}
				\deg_F u + \deg_F v - |\Delta_{F}(uv)| - 2, &\text{if } w = x,\\
				\deg_F w - 1, & \text{if } w \in \Delta_{F}(uv), \\
				\deg_F w, & \text{otherwise}.
			\end{cases}
			\]
			The desired equation now follows readily from the Handshaking Lemma.
		\end{claimproof}
		
		Hence, all summation indices $(F', L)$ in \cref{eq:del-contr} such that $F' \oslash L \cong K$ must satisfy the following:
		\begin{enumerate}
			\item $L = \{e'\}$ for some edge $e' \in E(F)$.
			
			This is immediate from $|V(K)| = |V(F)| - 1$ and $V(F) = V(F')$.
			\item $F' = F$. 
			
			Indeed, if $F' \oslash \{e'\} \cong K$ for some $F' \subseteq F$ with $e' \in E(F')$ and $V(F') = V(F)$
			then, by \cref{obs:edges-contract},
			$|E(F')| - 1 = |E(F' \oslash \{e'\})| + |\Delta_{F'}(e')| \geq
			|E(F' \oslash \{e'\})| = |E(K)| =  |E(F)| - 1$ and thus $E(F') = E(F)$.
		\end{enumerate}
		Each of these summation indices contributes $(-1)^{|E(F')|} = (-1)^{|E(F)|}$ to the coefficient of $\hom(K, -)$ in the linear combination ranging over non-isomorphic graphs.
		Hence, this coefficient is non-zero.
		By \cref{lem:lincomb}, $K \in \cl(\mathcal{F})$.
		
		Hence, $\cl(\mathcal{F})$ is closed under deleting and contracting edges. By \cref{lem:minors}, it is closed under taking minors.
		The implication \ref{rcl4} $\Rightarrow$ \ref{rcl2} follows from \ref{rcl1} $\Rightarrow$ \ref{rcl2} for $\cl(\mathcal{F})$ since $\equiv_{\mathcal{F}}$ and $\equiv_{\cl(\mathcal{F})}$ coincide.
	\end{proof}

	\subsection{Taking Subgraphs and Preservation under Full Complements}
	\label{sec:full-complements}
	
	\Cref{thm:full-complement}, which relates the property of a graph class $\mathcal{F}$ to be closed under taking subgraphs to the property of $\equiv_{\mathcal{F}}$ to be preserved under taking full complements,
	can now be extracted from the insights in \cref{sec:complements}.
	Since our definition of the homomorphism distinguishing closure involves only simple graph in order to be aligned with \cite{roberson_oddomorphisms_2022}, \cref{thm:full-complement} deviates slightly from the other results in \cref{tab:overview}.
	This is because the relations $\equiv_{\mathcal{F}}$ and $\equiv_{\cl(\mathcal{F})}$ a priori coincide only on simple graphs and not necessarily on all graphs, a crucial point raised by a reviewer.
	
	\begin{theorem} \label{thm:full-complement}
		For a graph class $\mathcal{F}$ and the assertions
		\begin{enumerate}
			\item $\mathcal{F}$ is closed under deleting edges,\label{cl1}
			\item $\equiv_{\mathcal{F}}$ is \emph{preserved under taking full complements}, i.e.\@ for all simple graphs $G$ and $H$ it holds that $G \equiv_{\mathcal{F}} H$ if and only if $\widehat{G} \equiv_{\mathcal{F}} \widehat{H}$,\label{cl2}
			\item $\cl(\mathcal{F})$ is closed under deleting edges,\label{cl3}
			\item $\cl(\mathcal{F})$ is \emph{closed under taking subgraphs}, i.e.\@ it is closed under deleting edges and vertices,\label{cl4}
			\item $\equiv_{\cl(\mathcal{F})}$ is preserved under taking full complements,\label{cl5}
		\end{enumerate}
		the implications \ref{cl1} $\Rightarrow$ \ref{cl2} $\Rightarrow$ \ref{cl3} $\Leftrightarrow$ \ref{cl4} $\Leftrightarrow$ \ref{cl5} hold.
	\end{theorem}

	\begin{proof}Assuming \ref{cl1}, by \cref{eq:complement}, for all not necessarily simple graphs $G$ and $H$, if $G \equiv_{\mathcal{F}} H$ and $\mathcal{F}$ is closed under deleting edges it holds that $\hom(F, \widehat{G}) = \hom(F, \widehat{H})$ for all $F \in \mathcal{F}$ as all graphs $F'$ appearing in \cref{eq:complement} are obtained from $F$ by deleting edges. 
		Hence, \ref{cl2} holds.
		
		Suppose \ref{cl2} holds and let $F \in \cl(\mathcal{F})$.
Since \cref{lem:lincomb} is not directly applicable since the condition in \ref{cl2} involves graphs with loops, we first prove the following claim:
		
		\begin{claim} \label{cl:full-complement}
For all graph $G$ and $H$ with loops at every vertex, if $G \equiv_{\mathcal{F}} H$ then $\hom(F, G) = \hom(F, H)$.
		\end{claim}
		\begin{claimproof}
			Let $K$ be an arbitrary graph with a loop at every vertex. Then
			$G \times K$ and $H \times K$ are graph with loops at every vertex.
			Hence, their full complements $\widehat{G \times K}$ and $\widehat{H \times K}$ are simple graphs.
			
			By \cref{eq:product}, if $G \equiv_{\mathcal{F}} H$ then 
			$G \times K \equiv_{\mathcal{F}} H \times K$.
			By \ref{cl2}, $\widehat{G \times K} \equiv_{\mathcal{F}} \widehat{H \times K}$, which implies that $\hom(F, \widehat{G \times K}) = \hom(F, \widehat{H \times K})$ since $\equiv_{\mathcal{F}}$ and $\equiv_{\cl(\mathcal{F})}$ coincide on simple graphs.
			By \cref{eq:complement,eq:product}, this in turn implies that
			\begin{equation} \label{eq:full-complement}
				\sum_{\substack{F' \subseteq F \\ V(F') = V(F)}} (-1)^{|E(F')|} \hom(F', G)\hom(F', K)  = \sum_{\substack{F' \subseteq F \\ V(F') = V(F)}} (-1)^{|E(F')|} \hom(F', H)\hom(F', K).
			\end{equation}

			Let $n \coloneqq |V(F)|$ and write $\mathcal{L}$ for the set of all simple graphs on at most $n$ vertices. We claim that the matrix $(\hom(F, G^\circ)))_{F, G \in \mathcal{L}}$ is invertible. Indeed, by \cref{lem:looping},
			\[
			\hom(F, G^\circ) = \sum_{L \subseteq E(F)} \hom(F \oslash L, G)
			= \sum_{F' \in \mathcal{L}} |\{L \subseteq E(F) \mid F \oslash L \cong F'\}| \hom(F', G).
			\]
			By classical arguments \cite{lovasz_operations_1967}, the matrix $(\hom(F, G))_{F, G \in \mathcal{L}}$ is invertible. When ordering the elements of $\mathcal{L}$ first by number of vertices and then by number of edges, the matrix $(|\{L \subseteq E(F) \mid F \oslash L \cong F'\}|)_{F, F' \in \mathcal{L}}$ is upper triangular and all its diagonal entries are $1$ since $F \oslash L \cong F$ iff $L = \emptyset$. Hence, $(\hom(F, G^\circ)))_{F, G \in \mathcal{L}}$ is invertible as the product of two invertible matrices.
			As in the proof of \cref{lem:lincomb},
			one may multiply \cref{eq:full-complement} with the inverse of this matrix to conclude that $\hom(F, G) = \hom(F, H)$.
		\end{claimproof}
		
		By \ref{cl2}, for all simple graphs $G$ and $H$, the assumption $G \equiv_{\mathcal{F}} H $ implies that $\widehat{G} \equiv_{\mathcal{F}} \widehat{H}$.
		Hence, by \cref{cl:full-complement}, $\hom(F, \widehat{G}) = \hom(F, \widehat{H})$  for all $F \in \cl(\mathcal{F})$.
		Finally, by \cref{eq:complement},
		\[
			\sum_{\substack{F' \subseteq F \\ V(F') = V(F)}} (-1)^{|E(F')|} \hom(F', G)
			= \sum_{\substack{F' \subseteq F \\ V(F') = V(F)}} (-1)^{|E(F')|} \hom(F', H).
		\]
			Since $G$ and $H$ are simple, we are in the setting of \cref{lem:lincomb}.
		Let $K$ be any graph obtained from $F$ by deleting edges, i.e.\@ $V(K) = V(F)$ and $E(K) \subseteq E(F)$.
		Each summation index $F'$ in \cref{eq:complement}  such that $F' \cong K$ contributes $(-1)^{|E(F')|} = (-1)^{|E(K)|}$.
		Hence, the coefficient of $\hom(K, -)$ in the linear combination obtained from the one above by grouping isomorphic summation indices is non-zero.
		By \cref{lem:lincomb}, it holds that $F' \in \cl(\mathcal{F})$.
		
		The implication \ref{cl3} $\Rightarrow$ \ref{cl4} follows from \cref{lem:minors}.
		The implication \ref{cl4} $\Rightarrow$ \ref{cl5} follows from \ref{cl1}~$\Rightarrow$~\ref{cl2}.
		The final implication \ref{cl5} $\Rightarrow$ \ref{cl3} follows from \ref{cl2} $\Rightarrow$ \ref{cl3} observing that $\cl(\cl(\mathcal{F})) = \cl(\mathcal{F})$.
	\end{proof}

	\subsection{Taking Induced Subgraphs, Contracting Edges, and Lexicographic Products}
	
	In this section, it is shown that a homomorphism distinguishing closed graph class is closed under taking induced subgraphs (contracting edges) if and only if its homomorphism indistinguishability relation is preserved under lexicographic products with a fixed graph from the left (from the right).
	
	Examples for equivalence relations preserved under lexicographic products related to chromatic graph parameters are listed in \cref{cor:no-hom}. 
	Further examples are of model-theoretic nature.
	It is for example easy to see that winning strategies of the Duplicator player in Hella's bijective pebble game~\cite{hella_logical_1996} can be composed along lexicographic products.
	
	A correspondence between taking induced subgraphs and left lexicographic products is established by the following \cref{prop:lexprod-indsub}:
	
	\begin{theorem} \label{prop:lexprod-indsub}
		For a graph class $\mathcal{F}$ and the assertions
		\begin{enumerate}
			\item $\mathcal{F}$ is closed under taking induced subgraphs,\label{indsub1}
			\item $\equiv_{\mathcal{F}}$ is \emph{preserved under left lexicographic products}, i.e.\@ for all simple graphs $G$, $H$, and $H'$, if $H \equiv_{\mathcal{F}} H'$ then $G \cdot H \equiv_{\mathcal{F}} G \cdot H'$,\label{indsub2}
			\item $\cl(\mathcal{F})$ is closed under taking induced subgraphs.\label{indsub3}
		\end{enumerate}
		the implications \ref{indsub1} $\Rightarrow$ \ref{indsub2} $\Leftrightarrow$ \ref{indsub3} hold.
	\end{theorem}

	An example for a lexicographic product from the right is the \emph{$n$-blow-up} $G \cdot \overline{K_n}$ of a graph~$G$.
	It is well-known \cite{lovasz_large_2012} that every homomorphism indistinguishability relation is preserved under blow-ups. 
	Preservation under arbitrary lexicographic products from the right, however, is a non-trivial property corresponding to the associated graph class being closed under edge contractions:

	\begin{theorem} \label{prop:lexprod-contract}
		For a graph class $\mathcal{F}$ and the assertions
		\begin{enumerate}
			\item $\mathcal{F}$ is closed under edge contractions,\label{contr1}
			\item $\equiv_{\mathcal{F}}$ is \emph{preserved under right lexicographic products}, i.e.\@  for all simple graphs $G$, $G'$, and $H$, if $G \equiv_{\mathcal{F}} G'$ then $G \cdot H \equiv_{\mathcal{F}} G' \cdot H$,\label{contr2}
			\item $\cl(\mathcal{F})$ is closed under edge contractions.\label{contr3}
		\end{enumerate}
		the implications \ref{contr1} $\Rightarrow$ \ref{contr2} $\Leftrightarrow$ \ref{contr3} hold.
	\end{theorem}

	For the proofs of \cref{prop:lexprod-contract,prop:lexprod-indsub}, homomorphism counts $\hom(F, G \cdot H)$ are written as linear combinations of homomorphism counts $\hom(F', H)$ where $F'$ ranges over induced subgraphs of $F$ in the case of \cref{prop:lexprod-indsub} and over graphs obtained from $F'$ by contracting edges in the case of \cref{prop:lexprod-contract}.
	The following succinct formula for counts of homomorphisms into a lexicographic product may  be of independent interest.
	
	\begin{theorem} \label{thm:lexprod-hom}
		Let $F$, $G$, and $H$ be simple graphs. Then
		\[
			\hom(F, G \cdot H) = \sum_{\mathcal{R}}
			\hom(F/\mathcal{R}, G) \hom(\coprod_{R \in \mathcal{R}} F[R], H)
		\]
		where the outer sum ranges over all $\mathcal{R} \in \Pi(V(F))$ such that $F[R]$ is connected for all $R \in \mathcal{R}$.
	\end{theorem}
	\begin{proof}The proof is by constructing a bijection between the set of homomorphisms $F \to G \cdot H$ and the set of triples $(\mathcal{R}, g, h)$ where $\mathcal{R} \in \Pi(V(F))$ is such that all $F[R]$, $R \in \mathcal{R}$, are connected, and $g \colon F/\mathcal{R} \to G$ and $h \colon \coprod_{R \in \mathcal{R}} F[R] \to H$ are homomorphisms.
		To that end, write $\pi_G \colon V(G \cdot H) \to V(G)$ and $\pi_H \colon V(G \cdot H) \to V(H)$ for the projection maps. 
		
		Let $f \colon F \to G \cdot H$ be a homomorphism. Define a partition $\mathcal{R}'$ of $V(F)$ with classes $(\pi_G \circ f)^{-1}(v)$ for $v \in V(G)$ and $\mathcal{R} \leq \mathcal{R}'$ as the coarsest partition whose classes $R \in \mathcal{R}$ all induce connected subgraphs $F[R]$, i.e.\@ for every $R' \in \mathcal{R}'$, the partition $\mathcal{R}$ contains one class for every connected component of $F[R']$.
		
		The homomorphism $g \colon F/\mathcal{R} \to G$ is given by the map sending $R \in \mathcal{R}$ to $(\pi_G \circ f)(v)$ for any $v \in R$. By definition of $\mathcal{R}$, this map is well-defined.
		It is indeed a homomorphism since if $R_1, R_2 \in \mathcal{R}$ are adjacent in $F/\mathcal{R}$,  there exist $x_1 \in R_1$ and $x_2 \in R_2$ such that $x_1x_2 \in E(F)$. 
		Observe that $R_1 \neq R_2$ because $F/\mathcal{R}$ is simple and hence $x_1$ and $x_2$ lie in different classes of $\mathcal{R}'$.
		As $(\pi_G \circ f)(x_1) \neq (\pi_G \circ f)(x_2)$, it holds that $(\pi_G \circ f)(x_1)$ and $(\pi_G \circ f)(x_2)$ are adjacent vertices in $G$ because $f$ is a homomorphism. Hence, $g(R_1)$ and $g(R_2)$ are adjacent in $G$.
		
		The homomorphism $h \colon \coprod_{R \in \mathcal{R}} F[R] \to H$ is given by $\pi_H \circ f$.
		This is indeed a homomorphism since if $x_1, x_2 \in V(F)$ are adjacent in $\coprod_{R \in \mathcal{R}} F[R]$ then they lie in the same class of $\mathcal{R}'$ and hence $h(x_1)h(x_2)$ is an edge of $H$.
		
		For injectivity, suppose that $f, f' \colon F \to G \cdot H$ are both mapped to $(\mathcal{R}, g, h)$. Then, $\pi_H \circ f = h = \pi_H \circ f'$ and furthermore for every $v \in V(G)$ with $v \in R$ for some $R \in \mathcal{R}$, $(\pi_G \circ f) (v) = g(R) = g'(R) = (\pi_G \circ f)(v)$. Hence, $f = f'$. 
		
		For surjectivity, let $(\mathcal{R}, g, h)$ be a triple where $\mathcal{R} \in \Pi(V(F))$ is such that all $F[R]$, $R \in \mathcal{R}$, are connected, and $g \colon F/\mathcal{R} \to G$ and $h \colon \coprod_{R \in \mathcal{R}} F[R] \to H$ are homomorphisms.
		Define $f \colon F \to G \cdot H$ by $v \mapsto ((g \circ \rho) (v), h(v))$ where $\rho \colon F \to F/\mathcal{R}$ is the map sending $v \in V(F)$ to $R \in \mathcal{R}$ such that $v \in R$.
		
		The map $f$ is a homomorphism. Indeed, for $uv \in E(F)$, distinguish cases: If $(g \circ \rho)(u) = (g \circ \rho)(v)$ then $\rho(u) = \rho(v)$ since otherwise $\rho(u)$ and $\rho(v)$ are adjacent in $F/\mathcal{R}$ and $(g \circ \rho)(u) \neq (g \circ \rho)(v)$ as $g$ is a homomorphism into a loop-less graph.
		In this case, $u, v \in R$ for some $R \in \mathcal{R}$ and thus $h(u)h(v)$ is an edge of $H$.
		If $(g \circ \rho)(u) \neq (g \circ \rho)(v)$ then in particular $\rho(u) \neq \rho(v)$, $\rho(u)\rho(v)$ is an edge of $F/\mathcal{R}$ and hence $(g \circ \rho)(u)$ and $(g \circ \rho)(v)$ are adjacent in $G$.
		In any case, $f(u)f(v)$ is an edge of $G \cdot H$.
		
		Write $(\mathcal{R}', g', h')$ for the image of $f$ under the aforementioned construction. It is claimed that $(\mathcal{R}', g', h') = (\mathcal{R}, g, h)$.
		
		To argue that $\mathcal{R} = \mathcal{R}'$, let $x_1 \dots x_\ell$ be a path in $F$ such that $(\pi_G \circ f)(x_1) = \dots = (\pi_G \circ f)(x_\ell)$.
		Then $(g \circ \rho)(x_1) = \dots = (g \circ \rho)(x_\ell)$ by construction. It has to be shown that $\rho(x_1) = \dots = \rho(x_\ell)$. If $\rho(x_i) \neq \rho(x_{i+1})$ for some $1 \leq i < \ell$ then $\rho(x_i)$ and $\rho(x_{i+1})$ are adjacent in $F/\mathcal{R}$, which cannot be since $G$ is loop-less and both of these vertices have the same image under $g$. This implies that if $x, y $ are in the same class of $\mathcal{R}'$ then they are in the same class of $\mathcal{R}$.
		Conversely, let $x_1 \dots x_\ell$ be a path in $F$ such that all $x_1,\dots, x_\ell \in R$ for some $R \in \mathcal{R}$.
		Then $\rho(x_1) = \dots = \rho(x_\ell)$ and hence $(g \circ \rho)(x_1) = \dots = (g \circ \rho)(x_\ell)$. In particular, if $x, y $ are in the same class of $\mathcal{R}$ then they are in the same class under $\mathcal{R}'$.
		
		To argue that $g = g'$, note that for any $R \in \mathcal{R} = \mathcal{R}'$ with $v \in R$, $g'(R) = (\pi_G \circ f)(v) = (g \circ \rho)(v) = g(R)$.
		Finally, $h' = \pi_H \circ f = h$.
	\end{proof}
	
	\cref{thm:lexprod-hom} yields \cref{prop:lexprod-indsub,prop:lexprod-contract} as follows: 

	\begin{proof}[Proof of \cref{prop:lexprod-indsub}]
		That \ref{indsub1} implies \ref{indsub2} is immediate from \cref{thm:lexprod-hom}.
		
		Let $F \in \cl(\mathcal{F})$ and $U \subseteq V(F)$. In order to show that $F[U] \in \cl(\mathcal{F})$ assuming \ref{indsub2}, let $G = K_n$ where $n \coloneqq |V(F)|$.
		Observe that $\hom(F/\mathcal{P}, G) > 0$ for every $\mathcal{P} \in \Pi(V(F))$. For the discrete partition $\mathcal{D}$, $\coprod_{R \in \mathcal{D}} F[R] \cong nK_1$. By \cref{lem:lincomb}, $K_1 \in \cl(\mathcal{F})$. 
		Write $\mathcal{P}$ for the partition whose classes are formed by the connected components of $F[U]$ and singleton classes otherwise. Then $\coprod_{P \in \mathcal{P}} F[P] \cong F[U] + m K_1$ where $m \coloneqq |V(F) \setminus U|$. By \cref{lem:lincomb}, $F[U] + mK_1 \in \cl(\mathcal{F})$ which implies, as for example argued in \cref{lem:minors}, that $F[U] \in \cl(\mathcal{F})$.
		Hence, \ref{indsub2} implies \ref{indsub3}.
		
		The implication \ref{indsub3} $\Rightarrow$ \ref{indsub2} follows from \ref{indsub1} $\Rightarrow$ \ref{indsub2} for $\cl(\mathcal{F})$ since $\equiv_{\mathcal{F}}$ and $\equiv_{\cl(\mathcal{F})}$ coincide.
	\end{proof}
	
	\begin{proof}[Proof of \cref{prop:lexprod-contract}]
		That \ref{contr1} implies \ref{contr2} is immediate from \cref{thm:lexprod-hom}.
		
		Let $F \in \cl(\mathcal{F})$. It is to show that $F/\mathcal{P} \in \cl(\mathcal{F})$ for every $\mathcal{P} \in \Pi(V(F))$ such that all $F[P]$ are connected for $P \in \mathcal{P}$.
		To that end, write $n \coloneqq |V(F)|$ and let $H = K_n$.
		Observe that $\hom(F[U], H) > 0$ for every $U \subseteq V(F)$.
		Hence, the coefficient of $\hom(F/\mathcal{P}, -)$ in the linear combination from \cref{thm:lexprod-hom} is non-zero. By \cref{lem:lincomb}, $F/\mathcal{P} \in \cl(\mathcal{F})$.
		Hence, \ref{contr2} implies \ref{contr3}.
		
		The implication \ref{contr3} $\Rightarrow$ \ref{contr2} follows from \ref{contr1} $\Rightarrow$ \ref{contr2} for $\cl(\mathcal{F})$ since $\equiv_{\mathcal{F}}$ and $\equiv_{\cl(\mathcal{F})}$ coincide.
	\end{proof}
	
	\Cref{prop:lexprod-contract,prop:lexprod-indsub} can be combined to yield the following \cref{cor:lexprod}:
	
	\begin{corollary} \label{cor:lexprod}
		For a graph class $\mathcal{F}$ and the assertions
		\begin{enumerate}
			\item $\mathcal{F}$ is closed under taking induced subgraphs and edge contractions,\label{ccontr1}
			\item $\equiv_{\mathcal{F}}$ is such that for all simple  graphs $G$, $G'$, $H$, and $H'$, if $G \equiv_{\mathcal{F}} G'$ and $H \equiv_{\mathcal{F}} H'$ then $G \cdot H \equiv_{\mathcal{F}} G' \cdot H'$,\label{ccontr2}
			\item $\cl(\mathcal{F})$ is closed under taking induced subgraphs and edge contractions.\label{ccontr3}
		\end{enumerate}
		the implications \ref{ccontr1} $\Rightarrow$ \ref{ccontr2} $\Leftrightarrow$ \ref{ccontr3}  hold.
	\end{corollary}
	\begin{proof}Assuming \ref{ccontr1}, if $G \equiv_{\mathcal{F}} G'$ and $H \equiv_{\mathcal{F}} H'$ then $G \cdot H \equiv_{\mathcal{F}} G' \cdot H$ by \cref{prop:lexprod-contract}
		and $G' \cdot H \equiv_{\mathcal{F}} G' \cdot H'$ by \cref{prop:lexprod-indsub}.
		By transitivity, $G \cdot H \equiv_{\mathcal{F}} G' \cdot H'$ and \ref{ccontr2} holds. The implication \ref{ccontr2} $\Rightarrow$ \ref{ccontr3} is immediate from \cref{prop:lexprod-contract,prop:lexprod-indsub}.
		The implication \ref{ccontr3} $\Rightarrow$ \ref{ccontr2} follows from \ref{ccontr1} $\Rightarrow$ \ref{ccontr2} for $\cl(\mathcal{F})$ since $\equiv_{\mathcal{F}}$ and $\equiv_{\cl(\mathcal{F})}$ coincide.  
	\end{proof}

	As a final observation, the following \cref{lem:contr} relates the property of being closed under edge contractions to the other closure properties in \cref{tab:overview,fig:relationship}.

	\begin{lemma} \label{lem:contr}
		If a homomorphism distinguishing closed graph class $\mathcal{F}$ is closed under contracting edges then it is closed under taking summands.
	\end{lemma}
	\begin{proof}Let $F \in \mathcal{F}$. Since every homomorphism distinguishing closed graph class is closed under disjoint unions, cf.\@ \cref{eq:coproduct} and \cite{roberson_oddomorphisms_2022}, it suffices to show that every connected component $C$ of $F$ is in $\mathcal{F}$. Let~$m$ denote the number of connected components of~$F$. 
		By contracting all edges, the graph $mK_1$ can be obtained from~$F$.
		Hence, as argued in \cref{lem:minors}, $K_1 \in \mathcal{F}$. 
		Moreover, the graph $C + (m-1)K_1$ can be obtained from~$F$ by contracting all edges not in~$C$. This implies as in \cref{lem:minors} that $C \in \mathcal{F}$.
	\end{proof}

	\subsection{Applications}
	\label{sec:applications}
	
	As applications of \cref{thm:complement,thm:full-complement,prop:lexprod-indsub,prop:lexprod-contract}, we conclude, in the spirit of \cite{atserias_expressive_2021}, that certain equivalence relations on graphs cannot be homomorphism distinguishing relations.

	\begin{corollary} \label{cor:complement}
		Let $\mathcal{F}$ be a non-empty graph class such that one of the following holds:
		\begin{enumerate}
			\item $\equiv_{\mathcal{F}}$ is preserved under complements, cf.\@ \cref{thm:complement},
			\item $\equiv_{\mathcal{F}}$ is preserved under full complements, cf.\@ \cref{thm:full-complement},
			\item $\equiv_{\mathcal{F}}$ is preserved under left lexicographic products, cf.\@ \cref{prop:lexprod-indsub}, or
			\item $\equiv_{\mathcal{F}}$ is preserved under right lexicographic products, cf.\@ \cref{prop:lexprod-contract}.
		\end{enumerate}
		Then  $G \equiv_{\mathcal{F}} H$ implies that $|V(G)| = |V(H)|$ for all graphs $G$ and $H$.
	\end{corollary}
	\begin{proof}
		By \cref{thm:complement,thm:full-complement,prop:lexprod-indsub,prop:lexprod-contract}, $\mathcal{F}$ can be chosen to be closed under taking minors, subgraphs, induced subgraphs, or contracting edges.
		In any case, $K_1 \in \mathcal{F}$ as $\mathcal{F}$ is non-empty and hence $|V(G)| = \hom(K_1, G) = \hom(K_1, H) = |V(H)|$.
	\end{proof}

	As concrete examples, consider the following relations. 
	
	\begin{corollary} \label{cor:no-hom}
		There is no graph class $\mathcal{F}$ satisfying any of the following assertions for all graphs $G$ and $H$:
		\begin{enumerate}
			\item $G \equiv_{\mathcal{F}} H$ iff $a(G) = a(H)$ where $a$ denotes the order of the automorphism group,\label{cor:app1}
			\item $G \equiv_{\mathcal{F}} H$ iff $\alpha(G) = \alpha(H)$ where $\alpha$ denotes the size of the largest independent set,\label{cor:app2}
			\item $G \equiv_{\mathcal{F}} H$ iff $\omega(G) = \omega(H)$ where $\omega$ denotes the size of the largest clique,\label{cor:app3}
			\item $G \equiv_{\mathcal{F}} H$ iff $\chi(G) = \chi(H)$ where $\chi$ denotes the chromatic number.\label{cor:app4}
		\end{enumerate}
	\end{corollary}
	\begin{proof}
		The relation in \cref{cor:app1} is preserved under taking complements. By \cite[Theorem~1, Corollary p.\@~90]{geller_chromatic_1975}, the relations in \cref{cor:app2,cor:app4} are preserved under left lexicographic products. For \cref{cor:app3}, the same follows from \cite[Theorem~1]{geller_chromatic_1975} observing that $\overline{G} \cdot \overline{H} = \overline{G \cdot H}$ and $\omega(G) = \alpha(\overline{G})$.
		In each case, it is easy to exhibit a pair of graphs $G$ and $H$ in the same equivalence class with different number of vertices. By \cref{cor:complement}, none of the equivalence relations is a homomorphism indistinguishability relation.
	\end{proof}
	
	\section{Equivalences over Self-Complementary Logics}
	\label{sec:logic}
	
	In this section, \cref{thm:main1} is derived from \cref{thm:complement}.
	The theorem applies to self-complementary logics, of which examples are given subsequently.
	Finally, a result from graph minor theory is used to relate logics on graphs to quantum isomorphism.
	
	For convenience, we recall the following definitions from \cref{sec:intro}. A \emph{logic on graphs} \cite{barwise_extended_2017} is a  pair $(\mathsf{L}, \models)$ of a class $\mathsf{L}$ and a relation $\models$ comparing graphs and elements of $\mathsf{L}$ which is isomorphism-invariant, i.e.\@ satisfying that for all $\phi \in \mathsf{L}$ and graphs $G$ and $H$ such that $G \cong H$ it holds that $G \models \phi$ iff $H \models \phi$.
	When convenient, we omit the reference to $\models$ and denote $(\mathsf{L}, \models)$ by $\mathsf{L}$.
	Two graphs are \emph{$\mathsf{L}$-equivalent} if for all $\phi \in \mathsf{L}$ it holds that $G \models \phi$ iff $H \models \phi$.
	One may think of a logic on graphs as a mere collection of isomorphism-invariant graph properties. Every $\phi \in \mathsf{L}$ defines such property.
	This very general definition has to be strengthened only slightly in order to yield \cref{thm:main2}:
	
	\begin{definition} \label{def:self-complementary}
		A logic on graphs $(\mathsf{L}, \models)$ is \emph{self-complementary} if for all 		
		$\phi \in \mathsf{L}$ there is a $\overline{\phi} \in \mathsf{L}$ such that
		for all simple graphs $G$ it holds that $G  \models \phi$ if and only if $ \overline{G} \models \overline{\phi}$.
	\end{definition}
	
	\thmLogic*

	\begin{proof}It is shown that $\equiv_{\mathcal{F}}$ is preserved under taking complements in the sense of \cref{thm:complement}.
		Suppose $G \equiv_{\mathcal{F}} H$. By assumption, for all $\phi \in \mathsf{L}$ it holds that $G \models \phi$ iff $H \models \phi$ and hence, by self-complementarity,
		\[
		\overline{G} \models \phi 
		\iff
		G \models \overline{\phi}
		\iff
		H \models \overline{\phi}
		\iff
		\overline{H} \models \overline{\overline{\phi}}
		\iff
		\overline{H} \models \phi.
		\]
		Here, the penultimate equivalence holds since $H \models \phi$ if and only if $H \models \overline{\overline{\phi}}$ for all~$\phi$ and $H$ by the definition of self-complementarity observing that $\overline{\overline{H}} \cong H$.
		Thus, $\overline{G} \equiv_{\mathcal{F}} \overline{H}$.
		By \cref{thm:complement}, $\mathcal{F}' \coloneqq \cl(\mathcal{F})$ is minor-closed.
	\end{proof}

	In particular, by \cref{cor:complement}, all $\mathsf{L}$-equivalent graphs $G$ and $H$ must have the same number of vertices unless $\mathsf{L}$ is trivial in the sense that all graphs $G$ and $H$ are $\mathsf{L}$-equivalent.

	\subsection{Examples for Self-Complementary Logics}
	A first example of a self-complementarity logic is first-order logic $\mathsf{FO}$ over the signature of graphs $\{E\}$.
	In order to establish this property, 
	a formula $\overline{\phi} \in \mathsf{FO}$ has to be constructed for every $\phi \in \mathsf{FO}$ such that $G \models \phi$ iff $\overline{G} \models \overline{\phi}$ for all graphs $G$. Only subformulae $Exy$ require non-trivial treatment:
	
	\begin{definition} \label{def:fo-complement}
		For every $\phi \in \mathsf{FO}$, define $\overline{\phi} \in \mathsf{FO}$ inductively as follows:
			\begin{enumerate}
				\item if $\phi = Exy$ then $\overline{\phi} \coloneqq \neg Exy \land (x \neq y)$,
				\item if $\phi = \bot$ or $\phi = \top$ then $\overline{\phi} \coloneqq \phi$,
				 if $\phi = (x = y)$ then $\overline{\phi} \coloneqq \phi$,
				 if $\phi = \neg \psi$ then $\overline{\phi} \coloneqq \neg \overline{\psi}$,
				 if $\phi = \psi \land \chi$ then $\overline{\phi} \coloneqq \overline{\psi} \land \overline{\chi}$,
				 if $\phi = \psi \lor \chi$ then $\overline{\phi} \coloneqq \overline{\psi} \lor \overline{\chi}$,
				 if $\phi = \exists x.\ \psi$ then $\overline{\phi} \coloneqq \exists x.\ \overline{\psi}$, and
				 if $\phi = \forall x.\ \psi$ then $\overline{\phi} \coloneqq \forall x.\ \overline{\psi}$.
			\end{enumerate}
	\end{definition}
	
	\begin{lemma} \label{lem:semantics-complemtary-formula}
		Let $\phi \in \mathsf{FO}$ be a formula with $k \geq 0$ free variables.
		Then for all simple graphs~$G$ with $\boldsymbol{v} \in V(G)^k$ it holds that
		\[
		G, \boldsymbol{v} \models \phi \iff 
		\overline{G}, \boldsymbol{v} \models \overline{\phi}.
		\]
		In particular, $\mathsf{FO}$ is self-complementary.
	\end{lemma}
	\begin{proof}
		The proof is by induction on the structure of $\phi$.
		If $\phi$ is $Exy$, observe that $v_1v_2 \in E(G)$ if and only if $v_1v_2 \not\in E(\overline{G})$ and $v_1 \neq v_2$.
		In all other cases, the claim is purely syntactical. 
	\end{proof}
	
	\cref{lem:semantics-complemtary-formula} gives a purely syntactical criterion for a fragment $\mathsf{L} \subseteq \mathsf{FO}$ to be self-complementary. Indeed, if $\overline{\phi} \in \mathsf{L}$ as defined in \cref{def:fo-complement} for all $\phi \in \mathsf{L}$ then $\mathsf{L}$ is self-complementary. 
	Note that the operation in \cref{def:fo-complement} increases neither the number of variables nor affects the quantifiers in the formula.
	Thus, \cref{lem:semantics-complemtary-formula} automatically extends to fragments of $\mathsf{FO}$ defined by restricting the number of variables, order or number of quantifiers.
	For extensions of $\mathsf{FO}$, \cref{def:fo-complement} can be easily extended.
	This yields a rich realm of self-complementary logics, of which the following \cref{ex:logics}  lists only a selection.
	
	\begin{example} \label{ex:logics} \label{ex:fo-fragment} \label{ex:counting}
		The following logics on graphs are self-complementary. For every $k, d \geq 0$,
		\begin{enumerate}
			\item the $k$-variable and quantifier-depth-$d$ fragments $\mathsf{FO}^k$ and $\mathsf{FO}_d$ of $\mathsf{FO}$,
			\item first-order logic with counting quantifiers $\mathsf{C}$ and its $k$-variable and quantifier-depth-$d$ fragments $\mathsf{C}^k$ and~$\mathsf{C}_d$,
			\item inflationary fixed-point logic $\mathsf{IFP}$, cf.\@ \cite{grohe_descriptive_2017},
			\item second-order logic $\mathsf{SO}$ and its fragments monadic second-order logic $\mathsf{MSO}_1$, existential second-order logic $\mathsf{ESO}$, cf.\@ \cite{courcelle_monadic_1994,blass_model_2010}.
		\end{enumerate}
	\end{example}
		\cref{cor:complement} readily gives an alternative proof of \cite[Propositions~1 and~2]{atserias_expressive_2021}, which assert that neither $\mathsf{FO}^k$-equivalence nor $\mathsf{FO}_d$-equivalence are characterised by homomorphism indistinguishability relations.
	The logic fragments $\mathsf{C}^k$ and $\mathsf{C}_d$ are however characterised by homomorphism indistinguishability relations \cite{dvorak_recognizing_2010,grohe_counting_2020}.

	\subsection{Applications: Graph Minor Theory and Expressive Power}

	The final result of this section demonstrates how graph minor theory can yield insights into the expressive power of logics via \cref{thm:main1}.
	Subject to it are self-complementary logics which have a homomorphism indistinguishability characterisation and are stronger than $\mathsf{C}^k$ for every $k$, e.g.\@ they are capable of distinguishing CFI-graphs \cite{cai_optimal_1992}.
	It is shown that equivalence w.r.t.\@ any such logic is a sufficient condition for quantum isomorphism, an undecidable equivalence comparing graphs \cite{atserias_quantum_2019}.
		
	\begin{theorem} \label{cor:quantum}
		Let $(\mathsf{L}, \models)$ be a self-complementary logic on graphs for which there exists a graph class
		$\mathcal{F}$ such that two graphs $G$ and $H$ are homomorphism indistinguishable over $\mathcal{F}$ if and only if they are $\mathsf{L}$-equivalent.
		Suppose that for all $k \in \mathbb{N}$ there exist graphs $G$ and $H$ such that $G \equiv_{\mathsf{C}^k} H$ and $G \not\equiv_{\mathsf{L}} H$.
		Then all $\mathsf{L}$-equivalent graphs are quantum isomorphic.
	\end{theorem}
	\begin{proof}
		Contrapositively, it is shown that if there exist non-quantum-isomorphic $\mathsf{L}$-equivalent graphs then there exists a $k \in \mathbb{N}$
		such that $G \equiv_{\mathsf{C}^k} H \implies G \equiv_{\mathsf{L}} H$ for all $G$ and $H$.
		By~\cite{mancinska_quantum_2019,dvorak_recognizing_2010}, this statement can be rephrased in the language of homomorphism indistinguishability as $\mathcal{P} \not\subseteq \cl(\mathcal{F}) \implies \exists k \in \mathbb{N}.\ \mathcal{F} \subseteq \cl(\mathcal{TW}_k)$ where $\mathcal{P}$ denotes the class of all planar graphs and $\mathcal{TW}_k$ the class of all graphs of treewidth at most $k$.
		By \cref{thm:complement}, $\cl(\mathcal{F})$ is a minor-closed graph class. 
		By~\cite[(2.1)]{robertson_graph_1986}, cf.\@ \cite[Theorem~3.8]{nesetril_sparsity_2012}, if $\cl(\mathcal{F})$ does not contain all planar graphs then it is of bounded treewidth. Hence, there exists a $k \in \mathbb{N}$ such that $\mathcal{F} \subseteq \cl(\mathcal{F}) \subseteq \mathcal{TW}_k \subseteq \cl(\mathcal{TW}_k)$.
	\end{proof}

	\section{Classification of Homomorphism Distinguishing Closed Essentially Profinite Graph Classes}
	\label{sec:profinite}
	 
	The central result of this section is a classification of the homomorphism distinguishing closed graph classes which are in a sense finite.
	Since every homomorphism distinguishing closed graph class is closed under disjoint unions, 
	infinite graph classes arise naturally when studying the semantic properties of the homomorphism indistinguishability relations of finite graph classes. Nevertheless, the infinite graph classes arising in this way are \emph{essentially finite}, i.e.\@ they exhibit only finitely many distinct connected components.
	One may generalise this definition slightly by observing that all graphs $F$ admitting a homomorphism into some fixed graph $G$ have chromatic number bounded by the chromatic number of $G$.
	Thus, in order to make a graph class $\mathcal{F}$ behave much like an essentially finite class, it suffices to impose a finiteness condition, for every graph $K$, on the subfamily of all $K$-colourable graphs in $\mathcal{F}$.
	
	Formally, for a graph $F$, write $\Gamma(F)$ for the set of connected components of $F$. For a graph class~$\mathcal{F}$, define $\Gamma(\mathcal{F})$ as the union of the $\Gamma(F)$ where $F \in \mathcal{F}$.
	For a graph class $\mathcal{F}$ and a graph $K$, define $\mathcal{F}_K \coloneqq \{F \in \mathcal{F} \mid \hom(F, K) > 0 \}$, the set of $K$-colourable graphs in $\mathcal{F}$.
	
	\begin{definition} \label{def:profinite}
		A graph class $\mathcal{F}$ is \emph{essentially finite} if $\Gamma(\mathcal{F})$ is finite.
		It is \emph{essentially profinite} if $\mathcal{F}_K$ is essentially finite for every graph $K$.
	\end{definition}

	Clearly, every finite graph class is essentially finite and hence essentially profinite.
	Other examples for essentially profinite classes are the class of all cliques.
	They represent a special case of the following construction from \cite[Theorem~6.16]{roberson_oddomorphisms_2022}:
	For every $S \subseteq \mathbb{N}$, the family
	\begin{equation} \label{ex:cliques}
	\mathcal{K}^S \coloneqq \left\{K_{n_1} + \dots + K_{n_r} \mid r \in \mathbb{N}, \{n_1, \dots, n_r\} \subseteq S \right\}
	\end{equation}
	is essentially profinite.
	In particular, there are uncountably many such families of graphs.
	Note that one may replace the sequence of cliques $(K_n)_{n\in \mathbb{N}}$ in \cref{ex:cliques} by any other sequence of connected graphs $(F_n)_{n \in \mathbb{N}}$ such that the sequence of chromatic numbers $(\chi(F_n))_{n \in \mathbb{N}}$ takes every value only finitely often.
	
	Every graph $F$ of an essentially finite family $\mathcal{F}$ can be represented uniquely as vector $\vec{F} \in \mathbb{R}^{\Gamma(\mathcal{F})}$  whose $C$-th entry for $C \in \Gamma(\mathcal{F})$ is the number of times the graph $C$ appears as a connected component in $F$.
	The classification of the homomorphism distinguishing closed essentially profinite graph classes can now be stated as follows.

	\begin{theorem} \label{thm:profinite}
		For an essentially profinite graph class $\mathcal{F}$, the following are equivalent:
		\begin{enumerate}
			\item $\mathcal{F}$ is homomorphism distinguishing closed,\label{pf1}
			\item For every graph $K$,
			if $\vec{K} \in \spn \{\vec{F} \in \mathbb{R}^{\Gamma(\mathcal{F}_{K} \cup \{K\})} \mid F \in \mathcal{F}_{K}\}$ then $K \in \mathcal{F}$,
			\label{pf2}
			\item $\mathcal{F}_K$ is homomorphism distinguishing closed for every graph $K$.\label{pf3}
		\end{enumerate}
	\end{theorem}

	\cref{thm:profinite} directly implies \cref{thm:main3}. Indeed, if $\mathcal{F}$ is union-closed and closed under summands then $\Gamma(\mathcal{F}) \subseteq \mathcal{F}$ and every graph $K$ such that $\Gamma(K) \subseteq \Gamma(\mathcal{F})$ is itself in $\mathcal{F}$.
	In particular, \cref{thm:profinite} implies that all essentially profinite union-closed minor-closed graph classes are homomorphism distinguishing closed.
	For example, for every graph $G$, the union-closure of the class of minors of $G$ is homomorphism distinguishing closed, cf.\@ \cite[Question~4]{roberson_oddomorphisms_2022}.

	Towards proving \cref{thm:profinite}, we first make the following general observation:
	Considering essentially profinite graph classes is very natural in light of the following \cref{lem:reduction}. For every graph $K$, the subset $\mathcal{F}_K$ of $\mathcal{F}$ is the object prescribing whether $K \in \cl(\mathcal{F})$.
	
	\begin{lemma} \label{lem:reduction}
		Let $\mathcal{F}$ be a graph class and $K$ be a graph. 
		Then $K \in \cl(\mathcal{F})$ if and only if $K \in \cl(\mathcal{F}_K)$.
	\end{lemma}
	\begin{proof}The backward implication is immediate since $\mathcal{F}_K \subseteq \mathcal{F}$ and thus $\cl(\mathcal{F}_K) \subseteq \cl(\mathcal{F})$.
		Conversely, suppose that $K \in \cl(\mathcal{F})$. Let $G$ and $H$ be arbitrary graphs such that $G \equiv_{\mathcal{F}_K} H$.
		Then, by \cref{eq:product}, $G \times K \equiv_{\mathcal{F}} H \times K$ since $\hom(F, K) = 0$ for all $F \in \mathcal{F} \setminus \mathcal{F}_K$.
		By assumption, $\hom(K,G \times K) = \hom(K, H \times K)$, and therefore $\hom(K, G) = \hom(K, H)$ since $\hom(K, K) > 0$. Thus, $K \in \cl(\mathcal{F}_K)$.
	\end{proof}
	
	The proof of \cref{thm:profinite} is based on a generalisation of a result by Kwiecie{\'n}, Marcinkowski, and Ostropolski-Nalewaja~\cite{kwiecien_determinacy_2022}.
	They proved the following \cref{thm:kwiecien22} for finite graph classes.
	The extension to essentially finite graph classes does not require much additional work but might make core ideas appear more transparently.

	\begin{theorem}[Kwiecie{\'n}--Marcinkowski--Ostropolski-Nalewaja \cite{kwiecien_determinacy_2022}] \label{thm:kwiecien22}
		Let $\mathcal{F}$ be an essentially finite family of graphs and $K$ be a graph. 
		Suppose that $\hom(F, K) > 0$ for all $F \in \mathcal{F}$.
		Let $\mathcal{C} \coloneqq \Gamma(\mathcal{F} \cup \{K\})$.
		Then $K \in \cl(\mathcal{F})$ if and only if $\vec{K} \in \spn \{\vec{F} \in \mathbb{R}^{\mathcal{C}} \mid F \in \mathcal{F}\}$.
	\end{theorem}
	\begin{proof}
		For the backward direction, suppose that $\vec{K} \in \spn \{\vec{F} \in \mathbb{R}^{\mathcal{C}} \mid F \in \mathcal{F}\}$. 
		Observe that this implies that $\Gamma(\mathcal{F}) = \mathcal{C}$, i.e.\@ all connected components of $K$ appear as connected components of some $F \in \mathcal{F}$.
		Write $\vec{K} = \sum_{i=1}^r \alpha_i \vec{F_i}$ for some $\alpha_1, \dots, \alpha_r \in \mathbb{R}$ and $F_1, \dots, F_r \in \mathcal{F}$.
		Also write $K = \coprod_{C \in \mathcal{C}} \beta_C C$ for some $\beta_C \in \mathbb{N}$.
		Observe that $\beta_C = \vec{K}_C = \sum_{i=1}^r \alpha_i (\vec{F_i})_C$ for all $C \in \mathcal{C}$.
		
		Let $G$ and $H$ be graphs such that $G \equiv_{\mathcal{F}} H$. If $\hom(F, G) = 0 = \hom(F, H)$ for some $F \in \mathcal{F}$ then $\hom(K, G) = 0 = \hom(K, H)$ by the assumption that $\hom(F, K) > 0$ for all $F \in \mathcal{F}$. 
		Hence, it may be supposed that $\hom(F, G) = \hom(F, H) > 0$ for all $F \in \mathcal{F}$.
		This is crucial for ruling out division by zero in the following argument.
		Observe that
		\[
		\hom(F, -) = \prod_{C \in \mathcal{C} \text{ s.t. } \vec{F}_C \geq 1} \hom(C, -)^{\vec{F}_C}
		\]
		for all $F \in \mathcal{F} \cup \{K\}$. 
		It follows that $\hom(C, G) > 0$ and $\hom(C, H) > 0$ for all $C \in \Gamma(\mathcal{F}) = \mathcal{C}$. 
		Thus, in every step of the following calculation, the bases of all exponentials are positive integers.
		\begin{align*}
			\hom(K, G) 
			&= \prod_{C \in \mathcal{C}} \hom(C, G)^{\beta_C}
			= \prod_{C \in \mathcal{C}} \hom(C, G)^{\sum_{i=1}^r \alpha_i (\vec{F_i})_C}
			= \prod_{i=1}^r \prod_{C \in \mathcal{C}}  \hom(C, G)^{\alpha_i(\vec{F_i})_C} \\
			&= \prod_{i=1}^r \hom(F_i, G)^{\alpha_i}
			= \prod_{i=1}^r \hom(F_i, H)^{\alpha_i}
			= \hom(K, H).
		\end{align*}
		
		Conversely, pick a finite family of graphs $\mathcal{G}$ such that the matrix $M \coloneqq (\hom)|_{\mathcal{C} \times \mathcal{G}}$ is invertible, e.g.\@ in virtue of \cite[Proposition~5.44(b)]{lovasz_large_2012}. Consider the map
		\begin{align*}
			\Psi \colon \mathbb{R}^{\mathcal{G}} & \to \mathbb{R}^{\mathcal{C}} \\
			a &\mapsto \sum_{G \in \mathcal{G}} \hom(\mathcal{C}, G) a_G = M a.
		\end{align*}
		We think of $\mathbb{N}^{\mathcal{G}} \subseteq \mathbb{R}^{\mathcal{G}}$ as the space of instructions for constructing graphs as disjoint unions of elements in~$\mathcal{G}$. The vector $a \in \mathbb{N}^{\mathcal{G}}$ corresponds to the graph $\coprod_{G \in \mathcal{G}} a_G G$.
		The map $\Psi$ associates with such a graph its $\mathcal{C}$-homomorphism vector. In this way, $\mathbb{R}^{\mathcal{C}}$ may be thought of as the space of $\mathcal{C}$-homomorphism vectors.
		As a vector space isomorphism, $\Psi$ is a homeomorphism.
		
		Contrapositively, suppose that $\vec{K} \not\in \spn \{\vec{F} \in \mathbb{R}^{\mathcal{C}} \mid F \in \mathcal{F}\}$.
		Pick an integer vector $z \in \mathbb{Z}^{\mathcal{C}}$ such that $\langle z, \vec{F} \rangle = 0$ for all $F \in \mathcal{F}$ and $\langle z, \vec{K} \rangle \neq 0$.
		This can be done by Gram--Schmidt orthogonalisation applied to the rational vectors spanning $\spn \{\vec{F} \in \mathbb{R}^{\mathcal{C}} \mid F \in \mathcal{F}\}$ and to the rational vector $\vec{K}$.
		The resulting vector~$z'$ with rational entries is orthogonal to all $\vec{F}$, $F \in \mathcal{F}$ and has non-zero inner-product with $z$. The integer vector $z$ is then obtained from $z'$ by clearing denominators.
		
		Pick $p \in \Psi(\mathbb{Q}^\mathcal{G}_{>0}) \subseteq \mathbb{Q}^\mathcal{C}_{>0}$. 
		In what follows, the vector $p$ is perturbed in a direction depending on~$z$. Technical complications arise when proving that the perturbed vector can again be interpreted as an instruction for constructing a graph, i.e.\@ has a positive rational preimage under~$\Psi$.
		To that end, consider the continuous function $\phi \colon \mathbb{R}_{>0} \to \mathbb{R}_{>0}^\mathcal{C}$ which maps $t \mapsto (t^{z_C} p_C \mid C \in \mathcal{C})$. 
		\begin{claim} \label{thm:kwiecien22-claim1}
			There exists $t > 1$ such that $\phi(t) \in \Psi(\mathbb{Q}^\mathcal{G}_{>0})$.
		\end{claim}
		\begin{claimproof}
			As the image of an open set under a homeomorphism, the set $\Psi(\mathbb{R}^\mathcal{G}_{>0})$ is open in $\mathbb{R}^{\mathcal{C}}$.
			Hence, there exists $\epsilon > 0$ such that $B_\epsilon(p) \subseteq \Psi(\mathbb{R}^\mathcal{G}_{>0})$.
			Since $\phi^{-1}(B_\epsilon(p))$ is open, $\phi(1) = p$, and $\mathbb{Q}_{>0}$ is dense in $\mathbb{R}_{>0}$, there exists a rational $t > 1$ such that $\phi(t) \in B_\epsilon(p) \subseteq \Psi(\mathbb{R}^\mathcal{G}_{>0})$.
			Because $t$ and $p$ are rational and $z$ is integral, $\phi(t) \in \mathbb{Q}^{\mathcal{C}}$. The matrix $M$ from definition of $\Psi$ has integer entries and hence its preimages of rational vectors are rational. This implies that $\phi(t) \in \Psi(\mathbb{Q}_{> 0}^{\mathcal{G}})$.
		\end{claimproof}
		Write $p' \coloneqq \phi(t)$ for the $t$ whose existence is guaranteed by \cref{thm:kwiecien22-claim1}.
		Let $s, s' \in \mathbb{Q}_{> 0}^{\mathcal{G}}$ denote the preimages of $p$ and $p'$ under $\Psi$, respectively. There exists a natural number $\lambda \geq 1$ such that $\lambda s$ and $\lambda s'$ lie in $\mathbb{N}^{\mathcal{G}}$. Write $H$ and $H'$ for the structures obtained by interpreting $\lambda s$ and $\lambda s'$ as instructions for disjoint unions over elements in $\mathcal{G}$, i.e.\@ $H = \coprod_{G \in \mathcal{G}} (\lambda s)_G G$ and $H' = \coprod_{G \in \mathcal{G}} (\lambda s')_G G$. Then
		\[
		\hom(F, H) 
		= \prod_{C \in \mathcal{C}} \hom(C, H)^{\vec{F}_C}
		= \prod_{C \in \mathcal{C}} \Psi(\lambda s)_C^{\vec{F}_C}
		= \prod_{C \in \mathcal{C}} \left(\lambda \Psi(s)_C\right)^{\vec{F}_C}
		= \prod_{C \in \mathcal{C}} \left(\lambda p_C \right)^{\vec{F}_C}
		\]
		and, similarly,
		\begin{align*}
			\hom(F, H') 
			= \prod_{C \in \mathcal{C}} \Psi(\lambda s')_C^{\vec{F}_C}
			= \prod_{C \in \mathcal{C}} \left(\lambda \Psi(s')_C\right)^{\vec{F}_C}
			= \prod_{C \in \mathcal{C}} \left(\lambda p'_C \right)^{\vec{F}_C}
			= t^{\langle z, \vec{F} \rangle} \prod_{C \in \mathcal{C}} \left(\lambda p_C \right)^{\vec{F}_C}
		\end{align*}
		for all graphs $F$ which are disjoint unions of graphs in $\mathcal{C}$. In particular, $\hom(F, H) = \hom(F, H')$ for all $F \in \mathcal{F}$ but $\hom(K, H) \neq \hom(K, H')$.
		Thus $K \not\in \cl(\mathcal{F})$.
	\end{proof}

	Towards the proof of \cref{thm:profinite}, we collect the following lemmas:

	\begin{lemma}\label{lem:cl-col}
		For every graph class $\mathcal{F}$ and every graph $K$, 
		$\cl(\mathcal{F})_K \subseteq \cl(\mathcal{F}_K)$.
	\end{lemma}
	\begin{proof}
		Consider the following chain of inclusions:
		\begin{align*}
			\cl(\mathcal{F})_K 
			= \bigcup_{L \in \cl(\mathcal{F})_K} \{L\}
			\subseteq \bigcup_{L \in \cl(\mathcal{F})_K} \cl(\mathcal{F}_L)
			\subseteq \cl\left( \bigcup_{L \in \cl(\mathcal{F})_K} \mathcal{F}_L \right)
			\subseteq \cl(\mathcal{F}_K).
		\end{align*}
		The first inclusion follows from \cref{lem:reduction}. Indeed, if $L \in \cl(\mathcal{F})$ then $L \in \cl(\mathcal{F}_L)$.
		The second inclusion is implied by \cref{lem:intersection}.
		The third inclusion holds since if $F \in \mathcal{F}_L$ for $L \in \cl(\mathcal{F})_K$ then $\hom(F, K) > 0$.
	\end{proof}
	For essentially profinite graph classes, \cref{thm:kwiecien22} implies that the converse of \cref{lem:cl-col} holds.
	
	\begin{lemma} \label{lem:pf-col}
		For every essentially profinite $\mathcal{F}$, it holds that $\cl(\mathcal{F}_K) = \cl(\mathcal{F})_K$ for every graph $K$.
	\end{lemma}
	\begin{proof}
		That $\cl(\mathcal{F})_K \subseteq \cl(\mathcal{F}_K)$ is the assertion of \cref{lem:cl-col}.	
		Conversely, $\cl(\mathcal{F}_K) \subseteq \cl(\mathcal{F})$ since $\mathcal{F}_K \subseteq \mathcal{F}$. It remains to argue that every $F \in \cl(\mathcal{F}_K)$ is $K$-colourable.
		If $F \in \cl(\mathcal{F}_K)$  then, by \cref{thm:kwiecien22}, $\Gamma(F) \subseteq \Gamma(\mathcal{F}_K)$. In other words, all connected components of $F$ are $K$-colourable. This implies that $F$ is $K$-colourable. Hence, $\cl(\mathcal{F}_K) \subseteq \cl(\mathcal{F})_K$.
	\end{proof}
	
	The concludes the preparations for the proof of \cref{thm:profinite}.
	
	\begin{proof}[Proof of \cref{thm:profinite}]
		Suppose that $\mathcal{F}$ is homomorphism distinguishing closed.
		Let $K$ be a graph such that $\Gamma(K) \subseteq \Gamma(\mathcal{F}_K)$.
		If $\vec{K} \in \spn \{\vec{F} \in \mathbb{R}^{\Gamma(\mathcal{F}_{K})} \mid F \in \mathcal{F}_{K}\}$
		then \cref{thm:kwiecien22} applies. Hence, $K \in \cl(\mathcal{F})$ and thus $K \in \mathcal{F}$ since $\mathcal{F}$ is homomorphism distinguishing closed.
		
		Assuming \cref{pf2}, let $L \not\in \mathcal{F}$.
		It follows that $\vec{L} \not\in \spn\{\vec{F} \in \mathbb{R}^{\Gamma(\mathcal{F}_L \cup \{L\}) } \mid F \in \mathcal{F}_L\}$.
		Indeed, if $\Gamma(L) \subseteq \Gamma(\mathcal{F}_L)$ then this holds by \cref{pf2}. If $\Gamma(L) \not\subseteq \Gamma(\mathcal{F}_L)$ then $L$ has a connected component which is not in $\mathcal{F}_L$ and the claim follows readily.
		Hence, by \cref{thm:kwiecien22,lem:reduction}, $L \not\in \cl(\mathcal{F})$. So $\mathcal{F}$ is homomorphism distinguishing closed.
		
		The equivalence of \cref{pf1,pf3} follows from \cref{lem:pf-col}. Indeed, if $\mathcal{F}$ is homomorphism distinguishing closed then
		$\cl(\mathcal{F}_K) = \cl(\mathcal{F})_K = \mathcal{F}_K$ for every $K$.
		Conversely,
		\[
		\cl(\mathcal{F}) = \bigcup_{K} \cl(\mathcal{F})_K = \bigcup_{K} \cl(\mathcal{F}_K) = \bigcup_K \mathcal{F}_K = \mathcal{F}. \qedhere
		\]
	\end{proof}
	
	Finally, we deduce the following \cref{lem:def-epf,lem:def-ef} from \cref{thm:kwiecien22}:

	\begin{corollary} \label{lem:def-ef}
		If a graph class $\mathcal{F}$ is essentially finite then $\cl(\mathcal{F})$ is essentially finite.
	\end{corollary}
	\begin{proof}
				Let $K \in \cl(\mathcal{F})$. By \cref{lem:reduction}, $K \in \cl(\mathcal{F}_K)$. 
		By \cref{thm:kwiecien22}, $\vec{K} \in \spn \{\vec{F} \in \mathbb{R}^{\Gamma(\mathcal{F}_K \cup \{K\})} \mid F \in \mathcal{F}_K\}$. In particular, $\Gamma(K) \subseteq \Gamma(\mathcal{F}_K)$. Hence, $\Gamma(\cl(\mathcal{F})) = \bigcup_{K \in \cl(\mathcal{F})} \Gamma(K) \subseteq \bigcup_{K \in \cl(\mathcal{F})} \Gamma(\mathcal{F}_K) \subseteq \Gamma(\mathcal{F})$. Thus, $\cl(\mathcal{F})$ is essentially finite.
	\end{proof}

	Since the class of all graphs is not essentially profinite,
the following \cref{lem:def-epf} implies that no homomorphism indistinguishability relation of an essentially profinite graph class is as fine as isomorphism.

\begin{corollary} \label{lem:def-epf}
	If a graph class $\mathcal{F}$ is essentially profinite then $\cl(\mathcal{F})$ is essentially profinite.
\end{corollary}
\begin{proof}
	It has to be argued that for every $K$ the class $\cl(\mathcal{F})_K$ is essentially finite. By \cref{lem:cl-col}, $\cl(\mathcal{F})_K \subseteq \cl(\mathcal{F}_K)$ and the right hand-side is an essentially finite graph class by \cref{lem:def-ef}.
\end{proof}
	
	\subsection{Applications of \cref{thm:profinite}}
	
	To demonstrate the inner workings of condition~\cref{pf2} in \cref{thm:profinite}, we consider the following examples.
	The first example shows that not even the weakest closure property from \cref{fig:relationship} is shared by all homomorphism distinguishing closed families.
	The second example answers a question from \cite[p.\@~29]{roberson_oddomorphisms_2022} negatively: Is the disjoint union closure of the union of homomorphism distinguishing closed families homomorphism distinguishing closed?
	Finally, the second and third example illustrate that the inclusions in \cref{lem:intersection} can be proper.
	
	\begin{example} \label{ex1} \label{ex:union}
		Let $F_1$ and $F_2$ be connected non-isomorphic homomorphically equivalent graphs.
		\begin{enumerate}
			\item The class $\mathcal{F}_1 \coloneqq \{n(F_1 + F_2) \mid n \geq 1\}$ is homomorphism distinguishing closed and not closed under taking summands.
			\item For the homomorphism distinguishing closed $\mathcal{F}_2 \coloneqq \{n F_1 \mid n \geq 1\}$, 
			the disjoint union closure of  $\mathcal{F}_1 \cup \mathcal{F}_2$ is not homomorphism distinguishing closed,
			\item Let $\mathcal{F}_3 \coloneqq \{n_1F_1 + n_2F_2 \mid n_1 \geq n_2 \geq 1\}$ and $\mathcal{F}_4 \coloneqq \{n_1F_1 + n_2F_2 \mid n_2 \geq n_1 \geq 1\}$.
			Then $\cl(\mathcal{F}_3 \cap \mathcal{F}_4) \subsetneq \cl(\mathcal{F}_3) \cap \cl(\mathcal{F}_4)$.
		\end{enumerate}
	\end{example}
	\begin{proof}For the first example, observe that clearly $\Gamma(\mathcal{F}) = \{F_1, F_2\}$.
		To verify the assumptions of \cref{thm:profinite}, let $K \coloneqq n_1 F_1 + n_2 F_2$ for $n_1, n_2 \geq 0$. 
		It holds that $\spn\{\vec{F} \in \mathbb{R}^{\Gamma(\mathcal{F})} \mid F \in \mathcal{F}\} = \left\{ \lambda \left( \begin{smallmatrix}
			1\\1
		\end{smallmatrix} \right) \mid \lambda \in \mathbb{R} \right\}$ and thus $\vec{K}$ is in this space only if $n_1 = n_2$. In this case, however, $K \in \mathcal{F}$ and thus $\mathcal{F}$ is homomorphism distinguishing closed. It remains to observe that $F_1 \not\in \mathcal{F}$.
		
		In the second example, by \cref{thm:profinite}, $\mathcal{F}_2$ is homomorphism distinguishing closed. 
		The disjoint union closure of $\mathcal{F}_1 \cup \mathcal{F}_2$ is $\mathcal{F} \coloneqq \{n_1 F_1 + n_2 F_2 \mid n_1 \geq n_2 \geq 1\}$.
		Since $\hom(F_1, F_2) > 0$, it holds that $\mathcal{F}_{F_2} = \mathcal{F}$ and hence $\vec{F_2} \in \spn\{\vec{F} \in \mathbb{R}^{\Gamma(\mathcal{F})} \mid \vec{F} \in \mathcal{F}_{F_2}\} = \mathbb{R}^{\Gamma(\mathcal{F})}$. However, $F_2 \not\in \mathcal{F}$. By \cref{thm:profinite}, $\mathcal{F}$ is not homomorphism distinguishing closed.
		
		In the third example, by \cref{thm:profinite}, $\cl(\mathcal{F}_3) = \cl(\mathcal{F}_4) = \{n_1F_1 + n_2F_2 \mid n_1, n_2 \geq 1\}$.
		However, $\mathcal{F}_3 \cap \mathcal{F}_4 = \{n(F_1 + F_2) \mid n \geq 1\}$, which is homomorphism distinguishing closed by the first example.
	\end{proof}

	\subsection{Complexity of Homomorphism Indistinguishability over Essentially Profinite Graph Classes}
	
	From a computational perspective, the central open question on homomorphism indistinguishability regards the complexity and computability of the decision problem  $\HomInd(\mathcal{F})$. 
	For a fixed graph class $\mathcal{F}$, this problem asks to determine whether two input graphs $G$ and $H$ are  homomorphism indistinguishable over~$\mathcal{F}$.
	While $\HomInd(\mathcal{TW}_k)$ for the class~$\mathcal{TW}_k$ of all graphs of treewidth at most $k$ is polynomial-time~\cite{dvorak_recognizing_2010,cai_optimal_1992}, 
	$\HomInd(\mathcal{P})$ for the class~$\mathcal{P}$ of all planar graphs is undecidable~\cite{mancinska_quantum_2019}.
	Furthermore, $\HomInd(\mathcal{G})$ for the class~$\mathcal{G}$ of all graphs coincides with graph isomorphism, which poses well-known complexity-theoretic challenges~\cite{babai_graph_2016}.
	These examples illustrate that the complexity-theoretic landscape of the problems  $\HomInd(\mathcal{F})$ for graph classes $\mathcal{F}$ is diverse and far from being comprehensively understood.
	
	In this section, we show that for all essentially finite classes $\mathcal{F}$, $\HomInd(\mathcal{F})$ is in polynomial time. 
	For essentially profinite $\mathcal{F}$, the problem can be arbitrarily hard.
	
	\begin{theorem} \label{thm:finite}
		Let $\mathcal{F}$ be an essentially finite graph class. 
		Then there exists a finite graph class $\mathcal{F}'$ such that $G \equiv_{\mathcal{F}} H$ if and only if $G \equiv_{\mathcal{F}'} H$ for all graphs $G$ and $H$.
		In particular, there is a polynomial-time algorithm for $\HomInd(\mathcal{F})$.
	\end{theorem}
	\begin{proof}Construct $\mathcal{F}'$ by choosing for every $\Lambda \subseteq \Gamma(\mathcal{F})$ a
		finite set $F_1^\Lambda, \dots, F_\ell^\Lambda \in \mathcal{F}_L$ such that the 
		$\vec{F}_1^\Lambda, \dots, \vec{F}_\ell^\Lambda \in \mathbb{R}^{\Gamma(\mathcal{F})}$ span the finite-dimensional space $\spn\{\vec{F} \in \mathbb{R}^{\Gamma(\mathcal{F})} \mid F \in \mathcal{F}_L\}$
		where $L \coloneqq \coprod_{C \in \Lambda} C$ and taking the union over all graphs $F_1^\Lambda, \dots, F_\ell^\Lambda \in \mathcal{F}_L$ constructed in this way. By construction, $\mathcal{F}' \subseteq \mathcal{F}$. Thus, it suffices to show that $\mathcal{F} \subseteq \cl(\mathcal{F}')$.
		
		Let $K \in \mathcal{F}$ be a graph. 
		By \cref{lem:reduction}, $K \in \cl(\mathcal{F}')$ if and only if $K \in \cl(\mathcal{F}'_K)$.
		Let $L \coloneqq \coprod_{C \in \Lambda} C$ where $\Lambda \coloneqq \Gamma(K)$. Since $L$ and $K$ are homomorphically equivalent, it holds that $\mathcal{F}'_K = \mathcal{F}'_L$.
		By construction, $\vec{K} \in \spn \{\vec{F}_1^\Lambda, \dots, \vec{F}_\ell^\Lambda\} \subseteq \mathbb{R}^{\Gamma(\mathcal{F})}$ and, by projection, the analogous statement holds in $\mathbb{R}^{\Delta}$ for $\Delta \coloneqq \Gamma(\{K, F_1^\Lambda, \dots, F_\ell^\Lambda\})$.
		By \cref{thm:kwiecien22}, $K \in \cl(\{F_1^{\Lambda}, \dots, F_\ell^{\Lambda}\}) \subseteq \cl(\mathcal{F}')$.
	\end{proof}
	For essentially profinite $\mathcal{F}$, the problem $\HomInd(\mathcal{F})$ can be arbitrarily hard.
	For a set $S \subseteq \mathbb{N}$, write $\Mem(S)$ for the problem of deciding given $n \in \mathbb{N}$ whether $n \not\in S$.
	Here, $n$ is encoded unarily, i.e.\@ the size of an instance is $n$. 
	Note that (the complement of) any decision problem $L \subseteq \{0,1\}^*$ can be reduced to $\Mem(S)$ for some $S$ at the expense of an exponential increase in the size of the instances.
	Recall the definition of the graph class $\mathcal{K}^S$ from \cref{ex:cliques}.

	\begin{theorem} \label{thm:complexity-profinite}
		For every $S \subseteq \mathbb{N}$, the problem $\Mem(S)$ polynomial-time many-one reduces to $\HomInd(\mathcal{K}^S)$.
	\end{theorem}
	
	The proof of \cref{thm:complexity-profinite} is build on the following result from \cite{boeker_complexity_2019}.
	
	\begin{theorem}[{\cite[Corollary~14]{boeker_complexity_2019}}] \label{thm:boeker-cor14}
		For every $n \in \mathbb{N}$, 
		there exist graphs $G_n$ and $H_n$ of order $\leq \max\{1,2(n-1)\}$ such that for all $\ell \in \mathbb{N}$,
		\[
		\hom(K_\ell, G) \neq \hom(K_\ell, H) \iff \ell = n.
		\]
		Furthermore, $G_n$ and $H_n$ can be constructed in polynomial time in $n$.
	\end{theorem}
	
	\begin{proof}[Proof of \cref{thm:complexity-profinite}]
		Consider the following reduction from $\Mem(S)$ to $\HomInd(\mathcal{K}^S)$.
		Given $n \in \mathbb{N}$, construct $G_n$ and $H_n$ via \cref{thm:boeker-cor14}.
		Then $G_n \equiv_{\mathcal{K}^S} H_n$ if and only if $n \not\in S$.
	\end{proof}

	Although \cref{thm:complexity-profinite} implies that $\HomInd(\mathcal{F})$ for essentially profinite graph classes $\mathcal{F}$ can be arbitrarily hard, the graph classes $\mathcal{K}^S$ arising there are not very well-behaved from a graph-theoretic point of view.
	In light of Roberson's conjecture \cite{roberson_oddomorphisms_2022}, minor-closed graph classes are of special interest.
	The following \cref{lem:profinite-minor} shows that considering essentially profinite rather than essentially finite graph classes, in a sense, does not add any value for minor-closed graph classes.

	\begin{lemma} \label{lem:profinite-minor}
		Every essentially profinite minor-closed graph class $\mathcal{F}$ is essentially finite.
	\end{lemma}
	\begin{proof}Since the class of all graphs is not essentially profinite, there exists a number $t \in \mathbb{N}$ such that all $F \in \mathcal{F}$ do not contain the clique $K_t$ as a minor.
		By a weak form of Hadwiger's conjecture \cite{seymour_hadwigers_2016} proven by Wagner \cite{wagner_beweis_1964}, this implies that all $F \in \mathcal{F}$ are $2^{t-1}$-colourable.
		In particular, $\mathcal{F} = \mathcal{F}_{K_{2^{t-1}}}$ is essentially finite.
	\end{proof}

	\subsection{Cancellation Laws}
	
	The results summarised in \cref{tab:overview} give necessary conditions for an equivalence relation $\approx$ comparing graphs to be a homomorphism indistinguishability relation over a graph class with certain closure properties.
	They may be viewed as a step towards an axiomatic characterisation of homomorphism indistinguishability relations, similar to the characterisation~\cite{freedman_reflection_2007,lovasz_semidefinite_2009} of functions $f \colon \mathcal{G} \to \mathbb{N}$ on the set of all graphs $\mathcal{G}$ which are of the form $f = \hom(-, G)$ for some graph $G$.
	Such a characterisation should give sufficient and necessary criteria for an equivalence relation $\approx$ to be a homomorphism indistinguishability relation over some graph class.
	In this context, exploring abstract properties of equivalence relations comparing graphs and their connection to homomorphism indistinguishability is imperative.
	
	\begin{definition}
	Let $K$ be a graph. An equivalence relation $\approx$ comparing graphs \emph{admits $K$-cancellation} if for all graphs $G$ and $H$ it holds that
	\(
	G \times K \approx H \times K \implies G \approx H.
	\)
	\end{definition}
	Lovász~\cite{lovasz_cancellation_1971} proved that the isomorphism relation $\cong$ admits $K$-cancellation if and only if $K$ is non-bipartite.
	The following \cref{lem:cancellation} shows that $K$-cancellation of a homomorphism indistinguishability relation $\equiv_{\mathcal{F}}$ depends solely on the homomorphism distinguishing closure of the subclass $\mathcal{F}_K \subseteq \mathcal{F}$ of $K$-colourable graphs in $\mathcal{F}$.
	In fact, the proof is based on the observation that the equivalence relation $- \times K \equiv_{\mathcal{F}} - \times K$ coincides with the homomorphism indistinguishability relation $\equiv_{\mathcal{F}_K}$.
	
	\begin{lemma} \label{lem:cancellation}
		For a graph class $\mathcal{F}$ and a graph $K$, the following are equivalent:
		\begin{enumerate}
			\item $\equiv_{\mathcal{F}}$ admits $K$-cancellation,\label{canc1}
			\item $\mathcal{F} \subseteq \cl(\mathcal{F}_K)$,\label{canc2}
			\item $\cl(\mathcal{F}) = \cl(\mathcal{F}_K)$.\label{canc3}
		\end{enumerate}
	\end{lemma}
	\begin{proof}
		First observe that for graphs $G$ and $H$ the conditions $G \times K \equiv_{\mathcal{F}} H \times K$ and $G \equiv_{\mathcal{F}_K} H$ are equivalent.
		Indeed, if $G \times K \equiv_{\mathcal{F}} H \times K$ and $F \in \mathcal{F}_K$ then $\hom(F, G) = \hom(F, G \times K)/\hom(F, K) = \hom(F, H)$ and hence $G \equiv_{\mathcal{F}_K} H$.
		Conversely, if $G \equiv_{\mathcal{F}_K} H$ and $F \in \mathcal{F}$ then $\hom(F, G \times K) = \hom(F, G)\hom(F, K) = \hom(F, H) \hom(F, K) = \hom(F, G \times K)$ since $\hom(F, G) = \hom(F, H)$ if $\hom(F, K) \neq 0$.
		
		By the initial observation, \cref{canc1} is equivalent to the assertion that $G \equiv_{\mathcal{F}_K} H$ implies $G \equiv_{\mathcal{F}} H $ for all graphs $G$ and $H$. Hence, \cref{canc1,canc2} are equivalent.
		Since $\mathcal{F} \subseteq \cl(\mathcal{F})$ and $\cl(\mathcal{F}_K) \subseteq \cl(\mathcal{F})$, assertions \cref{canc2,canc3} are equivalent.
	\end{proof}
	
	There is a succinct characterisation of essentially profinite graph classes admitting $K$-cancellation.
	
	\begin{lemma} \label{cor:cancellation-pf}
		Let $\mathcal{F}$ be essentially profinite and let $K$ be a graph.
		Then $\equiv_{\mathcal{F}}$ admits $K$-cancellation if and only if all graphs in $\mathcal{F}$ are $K$-colourable.
	\end{lemma}
	\begin{proof}
		By \cref{lem:pf-col}, $\cl(\mathcal{F}_K) \subseteq \cl(\mathcal{F})_K$.
		By \cref{lem:cancellation}, if $\equiv_{\mathcal{F}}$ admits $K$-cancellation then $\mathcal{F} \subseteq \cl(\mathcal{F}_K)$ which implies that $\mathcal{F} = \mathcal{F}_K$.
		The converse is straightforward.
	\end{proof}
	
	Beyond essentially profinite graph classes, it is less clear when $\equiv_{\mathcal{F}}$ admits $K$-cancellation.
	If $K$ is bipartite then this depends solely on whether the graphs in $\mathcal{F}$ are $K$-colourable.

	\begin{lemma} \label{cor:bipartite-exception}
		For every graph class $\mathcal{F}$ and every bipartite graph $K$, $\cl(\mathcal{F}_K) =  \cl(\mathcal{F})_K$.
		In particular, $\equiv_{\mathcal{F}}$ admits $K$-cancellation if and only if all graphs in $\mathcal{F}$ are $K$-colourable.
	\end{lemma}
	\begin{proof}
		By \cite[Lemma~5.8]{roberson_oddomorphisms_2022}, the class $\mathcal{G}_{K_2}$ of all bipartite graphs is homomorphism distinguishing closed.
		By \cref{thm:profinite}, so is $\mathcal{G}_{K_1}$, the family of all graphs which are $K_1$-colourable, i.e.\@ the edge-less graphs.
		Hence, for $i \in \{1,2\}$, by \cref{lem:intersection},
		\[
		\cl(\mathcal{F}_{K_i}) = \cl(\mathcal{F} \cap \mathcal{G}_{K_i}) \subseteq \cl(\mathcal{F}) \cap \cl(\mathcal{G}_{K_i}) = \cl(\mathcal{F}) \cap \mathcal{G}_{K_i} = \cl(\mathcal{F})_{K_i}.
		\]
		The converse containment follows from \cref{lem:cl-col}.
		It remains to observe that if $K$ is bipartite then $\mathcal{F}_K$ equals either $\mathcal{F}_{K_1}$ or $\mathcal{F}_{K_2}$, depending on whether $K$ is edge-less.
	\end{proof}

	Under the additional assumption that $\mathcal{F}$ is closed under subdivisions, the following \cref{prop:subdiv} completes the characterisation of the graphs $K$ for which $\equiv_{\mathcal{F}}$ admits $K$-cancellation.
	For example, the quantum isomorphism relation $\equiv_{\mathcal{P}}$ \cite{mancinska_quantum_2019} and $k$-variable counting logic equivalence \cite{dvorak_recognizing_2010} admit $K$-cancellation if and only if $K$ is non-bipartite.
	\Cref{prop:subdiv} strengthens Lov\'asz' result \cite{lovasz_cancellation_1971} using an argument from \cite{dvorak_recognizing_2010}.
	
	\begin{theorem} \label{prop:subdiv}
		If $\mathcal{F}$ is closed under subdivisions then $\mathcal{F} \subseteq \cl(\mathcal{F}_{K})$ for every non-bipartite graph $K$.
		In particular, $\equiv_{\mathcal{F}}$ admits $K$-cancellation.
	\end{theorem}
	\begin{proof}
		The proof is by showing the contrapositive, i.e.\@ for graphs $G$ and $H$, if $G \not\equiv_{\mathcal{F}} H$  then $G \not\equiv_{\mathcal{F}_K} H$.
		If there is $F \in \mathcal{F}$ without edges such that $\hom(F, G) \neq \hom(F, H)$ then $G \not\equiv_{\mathcal{F}_K} H$ because $F \in \mathcal{F}_K$.
		Hence, it may be supposed that there is $F \in \mathcal{F}$ with edge $e \in E(F)$ such that $\hom(F, G) \neq \hom(F, H)$.
		By~\cite[Lemma~11]{dvorak_recognizing_2010}, there exists a graph $F'$ obtained from $F$ by replacing $e$ by a path of length at least three such that $\hom(F', G) \neq \hom(F', H)$.
		Since $K$ is non-bipartite, it contains an odd cycle $C_{2n+1}$ for some $n \in \mathbb{N}$.
		By subdividing sufficiently many edges, any graph can be turned into a $C_{2n+1}$-colourable graph and any such graph is $K$-colourable. Hence, $G \not\equiv_{\mathcal{F}_K} H$. 
		By \cref{lem:cancellation}, this implies that $\mathcal{F} \subseteq \cl(\mathcal{F}_K)$.
	\end{proof}

	\section{Conclusion}
	
	The main technical contribution of this work is a characterisation of closure properties of graph classes $\mathcal{F}$ in terms of preservation properties of their homomorphism indistinguishability relations $\equiv_{\mathcal{F}}$, cf.\@ \cref{tab:overview}.
	In consequence, a surprising connection between logical equivalences and homomorphism indistinguishability over minor-closed graph classes is established.
	In this way, results from graph minor theory are made available to the study of the expressive power of logics on graphs.
	Finally, a full classification of the homomorphism distinguishing closed graph classes which are essentially profinite is given. 
	Various open questions of \cite{roberson_oddomorphisms_2022} are answered by results clarifying the properties of homomorphism indistinguishability relations and of the homomorphism distinguishing closure.
	
	It is tempting to view the results in \cref{tab:overview} as instances of a potentially richer connection between graph-theoretic properties of $\mathcal{F}$ and \emph{polymorphisms} of $\equiv_{\mathcal{F}}$, i.e.\@ isomorphism-invariant maps $\mathfrak{f}$ sending tuples of graphs to graphs such that $\mathfrak{f}(G_1, \dots, G_k) \equiv_{\mathcal{F}} \mathfrak{f}(H_1, \dots, H_k)$ whenever $G_i \equiv_{\mathcal{F}} H_i$ for all $i \in [k]$.
	Recalling the algebraic approach to CSPs, cf.\@ \cite{barto_polymorphisms_2017}, one may ask what structural insights into $\mathcal{F}$ can be gained by considering polymorphisms of $\equiv_{\mathcal{F}}$.
	More concretely, can closure under topological minors be characterised in terms of some polymorphism?
	
	\paragraph*{Acknowledgements}
	I would like to thank David E.\@ Roberson, Louis Härtel, Martin Grohe, Gaurav Rattan, and Christoph Standke for fruitful discussions.
	
	I was supported by the German Research Foundation (\textsmaller{DFG}) within Research Training Group 2236/2 (\textsmaller{UnRAVeL}) and by the European Union (\textsmaller{ERC}, SymSim, 101054974).
	
	Views and opinions expressed are however those of the author(s) only and do not necessarily reflect those of the European Union or the European Research Council Executive Agency. Neither the European Union nor the granting authority can be held responsible for them.

\end{document}